\documentclass[12pt]{amsart}
\usepackage{tikz}
\usetikzlibrary{matrix}

\usepackage{amssymb}
\usepackage{verbatim}
\usepackage{array}
\usepackage{latexsym}
\usepackage{enumerate}
\usepackage{amsmath}
\usepackage{amsfonts}
\usepackage{amsthm}
\usepackage{color}
\usepackage[english]{babel}

\newtheorem{theorem}{Theorem}[section]
\newtheorem{proposition}[theorem]{Proposition}
\newtheorem{lemma}[theorem]{Lemma}
\newtheorem{corollary}[theorem]{Corollary}

{\theoremstyle{definition}
\newtheorem{definition}{Definition}}
\newtheorem{example}[theorem]{Example}
\newtheorem{rem}[theorem]{Remark}

\def\cC{\mathcal C}
\def\cD{\mathcal D}
\def\cE{\mathcal E}
\def\cF{\mathcal F}
\def\cG{\mathcal G}

\def\cH{\mathcal H}

\def\cR{\mathcal R}

\def\cX{\mathcal X}
\def\cY{\mathcal Y}
\def\cZ{\mathcal Z}

\def\Aut{\mbox{\rm Aut}}

\def\Aut{\mbox{\rm Aut}}

\def\Ker{\mbox{\rm Ker}}

\def\Alt{\mbox{\rm Alt}}

\newcommand{\PSL}{\mbox{\rm PSL}}
\newcommand{\SL}{\mbox{\rm SL}}
\newcommand{\PGL}{\mbox{\rm PGL}}

\newcommand{\PSU}{\mbox{\rm PSU}}
\newcommand{\PGU}{\mbox{\rm PGU}}

\newcommand{\aut}{\mbox{\rm Aut}}
\newcommand{\gal}{\mbox{\rm Gal}}

\newcommand{\ha}{{\textstyle\frac{1}{2}}}

\begin{document}
\title{Algebraic curves with automorphism groups of large prime order}
\thanks{This research was performed within the activities of  GNSAGA - Gruppo Nazionale per le Strutture Algebriche, Geometriche e le loro Applicazioni of Italian INdAM. 
 The second author  was supported by FAPESP-Brazil, grant 2017/18776-6.}
 \thanks{Nazar Arakelian is with the Centro de Matem\'atica, Computa\c{c}\~ao e Cogni\c{c}\~ao, Universidade Federal do ABC,   Santo Andr\'e, SP  09210-580, Brazil. \\ E-mail: n.arakelian@ufabc.edu.br}
 \thanks{Pietro Speziali is with the Instituto de Ci\^encias Matem\'aticas e de Computa\c{c}\~ao, Universidade de S\~ao Paulo, S\~ao Carlos, SP 13560-970, Brazil.\\ E-mail: pietro.speziali@icmc.usp.br}
 
 
 
 \author{Nazar Arakelian}
 
 \author{Pietro Speziali}
 
\begin{abstract}
Let $\cX$ be an algebraic curve of genus $g$ defined over an algebraically closed field $K$ of characteristic $p \geq 0$, and $q$ a prime dividing $|\aut(\cX)|$. We say that $\cX$ is a $q$-curve.  Homma proved that either $q \leq g+1$ or $q = 2g+1$, and classified $(2g+1)$-curves. In this note, we classify $(g+1)$-curves, and fully characterize the automorphism groups of $q$-curves for $q= 2g+1, g+1$. We also give some partial results on $q$-curves for $q = g, g-1$. 
\end{abstract}

\maketitle

\section{Introduction}\label{intr}

Let $\cX$ be a (projective, algebraic, non-singular, absolutely irreducible) curve of genus $g$ defined over an algebraically closed field of characteristic $p \geq 0$.  Let $q$ be a prime dividing the order of the automorphism group $\aut(\cX)$ of $\cX$. Homma \cite[Theorem 1]{homma1980} proved that either $q \leq g+1$ or $q = 2g+1$. In this paper, we study curves whose automorphism group is divided by a prime $q$ that is \emph{large} compared to $g$; more specifically, we are interested in the cases when $q \geq g-1$. A motivation (and at, the same time, a nice application of) for our results is the problem of determining, for any fixed genus $g$, the possibilities of the automorphism groups for curves of genus $g$, as well as, their respective models  up to birational equivalence, see for instance \cite{maagard, malmendier}. 

From now on, to avoid long periphrases and repetitions, we introduce the following (non-standard) terminology. 

\begin{definition}
Let $q$ be a prime number. A curve $\cX$ defined over an algebraically closed field $K$ of characteristic $p \geq 0$ is a $q$-curve if $\aut(\cX)$ contains a subgroup $C_q$ of order $q$. A $q$-curve is \emph{tame} if either $p = 0$ or $q \neq p$, \emph{wild} otherwise. 
\end{definition}

Here, two problems naturally arise. First, the problem of classifying $q$-curves for any fixed $q$ (possibly, as we will do, for $q = f(g)$, a function on the genus of $\cX$); second, once such curves are classified, to determine their full automorphism group. Homma  \cite[Theorem 2]{homma1980} classified all $(2g+1)$-curves (up to birational equivalence); later Seyama \cite[Theorem 3.3]{seyama} computed their automorphism group when $p=0$. 

The first problem we address (and to which Section \ref{2g+1} is devoted) is then the determination of the automorphism group of tame $(2g+1)$-curves in any characteristic (as wild $(2g+1)$-curves are always hyperelliptic). Let $\cX$ be a $(2g+1)$-curve; the most difficult part is to understand what happens  when $\aut(\cX)$ is non-tame. If this is the case, we show that $\aut(\cX)$ is a finite simple group. By combining Henn's classification of curves with more than $8g^3$ automorphisms and a result by  Vdovin \cite{vdovin} bounding the size of abelian subgroups of finite simple groups, we prove that  a cyclic group of order $2g+1$ is normal in $\aut(\cX)$ unless $\cX$ is birationally equivalent to  a Hermitian curve. 

Then, we turn our attention to $(g+1)$-curves. In Section \ref{g+1}, we first provide the classification of tame and wild $(g+1)$-curves. Then, we characterize hyperelliptic $(g+1)$-curves and finally, we provide the full automorphism group of tame, non-hyperelliptic $(g+1)$-curves whenever $p \neq 2,3$, as well as, the full automorphism group of wild $(g+1)$-curves. In the tame case, as $(g+1)$-curves have even genus, we apply the deep results obtained by Giulietti and Korchm\'aros in \cite{giulietti-korchmaros-2017}. We prove that the automorphism group of a $(g+1)$-curve $\cX$ has to be \emph{small}, in the sense that the Hurwitz bound holds and that a group of prime order equal to $(g+1)$ must always be normal in $\aut(\cX)$, with only one exception in genus $4$. 

Finally, in Section 5 we give some partial results on the classification of $q$-curves for $q =g, g-1$.  Such cases seem rather difficult and deserve a separate investigation.

\section{Background and preliminary results}\label{background}
Our notation and terminology are standard. For an exhaustive treatise of the theory of curves and algebraic function fields, the reader is referred to \cite{hirschfeld-korchmaros-torres2008} and \cite{stbook}. Let $\cX$ be a curve defined over an algebraically closed field $K$ of characteristic $p \geq 0$.  We denote by $K(\cX)$ the function field of $\cX$. By a point $P \in \cX$ we mean a point in a non-singular model of $\cX$; in this way, we have a one-to-one correspondence between points of $\cX$ and places of $K(\cX)$. 

Let $\aut(\cX)$ denote the full automorphism group of $\cX$. For a subgroup $G$ of $\aut(\cX)$, we denote by $K(\cX)^G$ the fixed field of $G$. A non-singular model $\bar{\cX}$ of  $K(\cX)^G$ is referred to as the quotient curve of $\cX$ by $G$ and denoted by $\cX/G$. The field extension $K(\cX):K(\cX)^G$ is Galois with Galois group $G$. For a point $P \in \cX$, $G(P)$ is the orbit of $P$ under the action of $G$ on $\cX$ seen as a point-set. The orbit $G(P)$ is said to be long if $|G(P)| = |G|$, short otherwise. There is a one-to-one correspondence between short orbits and ramified points in the extension $K(\cX):K(\cX)^G$. $G$ might have no short orbits; if this is the case, the cover $\cX \rightarrow \cX/G$ (or equivalently, the extension $K(\cX):K(\cX)^G$) is unramified. 

For $P \in \cX$, the subgroup $G_P$ of $G$ consisting of all elements of $G$ fixing $P$ is called the stabilizer of $P$ in $G$.  We will often refer to $G_P$ as to the \emph{1-point stabilizer} (or, sometimes, the \emph{one-point stabilizer}) of $G$. For a non-negative integer $i$, the $i$-th ramification group of $\cX$ at $P$ is denoted by $G_P^{(i)}$, and defined by
$$
G_P^{(i)}=\{\sigma \ | \ v_P(\sigma(t)-t)\geq i+1, \sigma \in G_P\}, 
$$
 where $t$ is a local parameter at $P$ and $v_P$ is the respective discrete valuation. Here $G_P=G_P^{(0)}$. Furthermore, $G_P^{(1)}$ is the unique Sylow $p$-subgroup of $G_P^{(0)}$, and the factor group $G_P^{(0)}/G_P^{(1)}$ is cyclic of order prime to $p$; see \cite[Theorem 11.74]{hirschfeld-korchmaros-torres2008}. In particular, if $|G_P|$ is a power of  $p$, then $G_P=G_P^{(0)}=G_P^{(1)}$. For a point $P \in \cX$, the ramification index of $P$ is defined as   $e_P := |G^{(0)}_P|$ and the different exponent of $P$ is $d_P := \sum_{i = 0}^{\infty}(|G_P^{(i)}|- 1)$. 
 
 Let $g$ and $\bar{g}$ be the genus of $\cX$ and $\bar{\cX}=\cX/G$, respectively. The Riemann-Hurwitz genus formula is 
\begin{equation}\label{rhg}
2g-2=|G|(2\bar{g}-2)+\sum_{P \in \cX}d_P
\end{equation}
see \cite[Theorem 11.72]{hirschfeld-korchmaros-torres2008}.
 If $\ell_1,\ldots,\ell_k$  are the sizes of the short orbits of $G$, then (\ref{rhg}) yields
\begin{equation}\label{rhso}
2g-2 \geq |G|(2\bar{g}-2)+\sum_{\nu=1}^{k} \big(|G|-\ell_\nu\big),
\end{equation}
and equality holds if $\gcd(|G_P|,p)=1$ for all $P \in \cX$; see \cite[Theorem 11.57 and Remark 11.61]{hirschfeld-korchmaros-torres2008}. 

We now state some further facts, as well as, give  a few definitions and notation that we are going to need to prove our results.

  \begin{theorem}\label{th:largeaut}\cite[Theorem 11.56]{hirschfeld-korchmaros-torres2008} Let $\cX$ be an irreducible curve of genus $g \geq 2$.
If $G$ is a $K$-automorphism group of $\cX$, then Hurwitz's upper bound $|G| \leq 84(g - 1)$
holds in general,  with exceptions occurring only in positive characteristic. Such exceptions can only occur when the quotient curve  $\cX/G$ is rational, and $G$ has at most three short orbits, as follows:
  \begin{enumerate}
  \item[{\rm(a)}]exactly three short orbits, two tame and one non-tame, with $p \geq 3$;
  \item[{\rm(b)}]exactly two short orbits, both non-tame;
 \item[{\rm(c)}]only one short orbit which is non-tame;
\item [{\rm(d)}]exactly two short orbits, one tame and one non-tame.
\end{enumerate}
  \end{theorem}

\begin{theorem}\label{Roq}(Roquette, \cite{Roq}) Let $\mathcal{X}$  be an  irreducible
curve of genus $g\geq 2$  defined over  a field of characteristic $p>g+1$.
Then  $|\Aut (\mathcal{X})|\leq 84(g-1)$ holds,  except for  the hyperelliptic 
 curve defined by the affine equation $Y^p-Y-X^2= 0$, with $g =\frac{1}{2}(p + 1)$ and $|\Aut (\mathcal{X})| = 2p(p^2-1)$.
\end{theorem}

The following result will be crucial in our paper. 

\begin{theorem}(Homma, \cite{homma1980})\label{homma1}
If a prime number $q$ is an order of an automorphism on an algebraic curve $\cX$ of genus $g \geq 2$, then either $q \leq g+1$, or $q=2g-1$ and $g=2$, or $q = 2g+1$. 
\end{theorem}

We will also need the following results. 

\begin{theorem}\cite[Theorem 11.127]{hirschfeld-korchmaros-torres2008} \label{henn}
Let $\cX$ be a curve of genus $g \geq 2$. If $|\aut(\cX)| \geq 8g^3$, then $\cX$ is birationally equivalent to one of the following:
\begin{enumerate}
\item[{\rm (I)}] the hyperelliptic curve of affine equation $Y^2+Y+X^{2k+1} = 0$ with $p=2, g=2k-1$; also, $\aut(\cX)$  fixes a point $P$ and $|\aut(\cX)| = 2^{2k+1}(2^k + 1)$;
\item[{\rm (II)}] the hyperelliptic curve of affine equation $Y^2 -(X^{p^h} -X) = 0$ with $ h \geq 1, p>2, g= (p^h-1)/2$; also, $G/M \cong \PSL(2,p^h) $or $G/M \cong \PGL(2,p^h)$, where $|M| = 2$; 
 \item[{\rm (III)}] the Hermitian curve of affine equation $Y^{p^h} + Y-X^{p^h+1} = 0$ with $p \geq 2, h>1, g = (p^{2h} -p^h)/2$; also, $\aut(\cX) \cong \PSU(3, p^h)$  or $\aut(\cX) \cong \PGU(3, p^h)$;
\item[{\rm (IV)}] the DLS curve  of affine equation $X^{2^h}(X^{2^{h+1}}+X)-(Y^{2^{h+1}}+Y) =0$ with $p=2, g = 2^h( 2^{h+1} - 1 )$; also, $\aut(\cX) \cong Sz (2^{h+1})$.
\end{enumerate}
\end{theorem}

Given a finite group $G$, the odd core of $G$, denoted by $O(G)$, is the maximal (with respect to inclusion) normal subgroup of odd order of $G$. A group $G$ is said to be odd core-free if $O(G)$ is trivial. 

\begin{theorem}(Giulietti and Korchm\'aros,  \cite{giulietti-korchmaros-2017})\label{thmgenuseven}
Let $p>2$. If $G$ is a subgroup of the automorphism group of some non-rational algebraic curve with even genus defined over an algebraically closed field of odd characteristic $p$, then, with $q$ being a prime power (with such prime not necessarily equal to $p$), one of the following cases occurs up to isomorphism:
\begin{enumerate}
\item[{\rm (i)}] $G$ has odd order;
\item[{\rm (ii)}] $G = O(G) \rtimes S_2 $, where $S_2$ is a $2$-group with a cyclic subgroup of index $2$;
\item[{\rm (iii)}] the commutator subgroup of $G/O(G)$ is isomorphic to $\SL(2,q)$ with $q \geq 5$;
\item [{\rm (iv)}] $\PSL(2,q) \leq G/O(G) \leq {\rm P\Gamma L}(2,q)$ with $q \geq 3$;
\item  [{\rm (v)}] $\PSL(3,q) \leq G/O(G) \leq {\rm P\Gamma L}(3,q)$ with $q \equiv 3 \mod 4$;
\item [{\rm (vi)}] $\PSU(3,q) \leq G/O(G) \leq {\rm P\Gamma U}(3,q)$ with $q \equiv 1 \mod 4$;
\item [{\rm (vii)}] $G/O(G) = {\rm Alt}_7$;
\item [{\rm (viii)}] $G/O(G) = {\rm M}_{11}$;
\item [{\rm (ix)}] $G/O(G) = {\rm GL}(2,3)$;
\item[{\rm (x)}] $G/O(G)$ is the unique perfect group of order 5040 and $(G/O(G))/Z(G) \cong {\rm Alt}_7$;
\item[{\rm (xi)}] $G/O(G)$ is the group of order 48 named $SmallGroup(48,28) $ in the GAP-database.
\end{enumerate}
\end{theorem}

\section{On $(2g+1)$-curves}\label{2g+1}
\subsection{Known results}
$(2g+1)$-curves were classified by Homma in  \cite[Theorem 1]{homma1980}. To make our paper as self-contained as possible, we state his classification without  proof. 
\begin{theorem}\label{thm:homma1980}
Let $2g+1$ be a prime number. 
\begin{enumerate}
\item[{\rm (a)}] A curve $\cX$  is a tame $(2g+1)$-curve if and only if $\cX$ is birationally equivalent to one of the following plane curves:
$$
\cX_{m,n}: Y^{2g+1} = X^{m-n}(X-1)^n \: (1\leq n < m \leq g+1).
$$
\item[{\rm (b)}] A curve $\cX$ is a wild $(2g+1)$-curve if and only if $\cX$ is birationally equivalent to the plane curve:
$$
\cR: Y^2 = X^{2g+1}-X. 
$$
\end{enumerate}
\end{theorem}

\begin{rem}
If $\cX$ is a tame $(2g+1)$-curve, then $C_{(2g+1)}$ fixes exactly three points, $P_0,P_1, P_\infty$; we shall denote by $\Omega$ the set of such fixed points. 
\end{rem}

Further results on $(2g+1)$-curves were later given by Seyama \cite{seyama}, whose main results we summarize in the following Theorem.

\begin{theorem}\label{2g+1:facts}
Let $\cX_{m,n}$ be a tame $(2g+1)$-curve. Then the following hold.
\begin{enumerate}
\item[{\rm (i)}] There exists an integer $r \in \{1,\ldots, 2g-1\}$ such that $\cX_{m,n}$ is birationally equivalent over $K$ to $\cX_{1,r}$. 
\item[{\rm (ii)}] $\cX_r := \cX_{1,r}$ is hyperelliptic if, and only if, $r \in \{1,g,2g+1\}$. 
\item[{\rm (iii)}] There exists an automorphism $\tau \in \aut(\cX_r)$ of order $3$ normalizing $G$ if, and only if, $r^2+r+1 \equiv 0 \mod (2g+1)$. 
\item[{\rm (iv)}] Let $p = 0$; then  $C_{(2g+1)}$ is normal in $\aut(\cX_r)$, whence $|\aut(\cX_r)| = c(2g+1)$, where $ c \in \{1,2,3\}$, unless $g = 3$ and $\cX_r$ is isomorphic to the Klein quartic. 
\end{enumerate}
\end{theorem}

\begin{rem}\label{rem:char}
The result in  Theorem \ref{2g+1:facts} (iv) holds whenever $p > 0$ and $\aut(\cX_r)$ is tame. 
\end{rem}

\subsection{Automorphisms of tame $(2g+1)$-curves for $p >0$}

Throughout this subsection, $\cX_r$ is a tame $(2g+1)$-curve of genus $g > 3$. Also, we let $C_{(2g+1)} = \langle \alpha \rangle$, and $\Omega = \{P_0,P_1,P_\infty\}$ is the set of points fixed by $\alpha$. By Theorem \ref{Roq} and  Remark \ref{rem:char}, we may assume $ 0 < p \leq g+1$. We aim to prove the following Theorem. 

\begin{theorem}\label{teo:fundaut2g+1}
$C_{(2g+1)}$ is normal in $\aut(\cX_r)$, whence $|\aut(\cX_r)| = c(2g+1)$, where $ c \in \{1,2,3\}$, unless $g= p^h(p^h-1)/2$ for some $h \geq 1$ and $\cX_r$ is isomorphic to the Hermitian curve $\cH: X^{p^h+1} = Y^{p^h}+Y$. 
\end{theorem}

To prove Theorem \ref{teo:fundaut2g+1}, we need several results. We start by proving some basic facts. 

\begin{lemma}\label{sylow}
$C_{(2g+1)}$ is a Sylow $(2g+1)$-subgroup of $\aut(\cX_{r})$. 
\end{lemma}
\begin{proof}  By contradiction, let $S$ be a Sylow $(2g+1)$-subgroup of $\aut(\cX_{r})$ with $|S| = (2g+1)^i$, $ i \geq 2$. Then $S$ has a subgroup $S_1$ of order $(2g+1)^2$. Then $S_1$ is abelian, a contradiction to  \cite[Theorem 11.79]{hirschfeld-korchmaros-torres2008}. 
\end{proof}

\begin{rem}\label{remark:s3}
Let $H$ be the normalizer of $C_{(2g+1)}$ in $\aut(\cX_r)$. Then  $H$ must act on the set $\Omega$, that is, there exists a permutation representation $\rho: H \rightarrow S_3$. A computation via the Riemann-Hurwitz formula shows that $\Ker(\rho)  = C_{(2g+1)}$. Hence, $H/C_{(2g+1)} \leq S_3 \leq \PGL(2,K)$.  
\end{rem}

\begin{lemma}\label{lem:2g+1hyp}
If $2 \mid |H|$, then $\cX_{r}$ is hyperelliptic. 
\end{lemma}
\begin{proof}
Let $i$ be an involution in $H$; then $i$ fixes a point in $\Omega$, say $P_\infty$, and permutes $P_0$ and $P_1$ by Remark \ref{remark:s3}. Let $p\neq 2$; then $\langle i, \alpha\rangle$ is cyclic.  Let $\bar{i} = H/C_{(2g+1)} $, and denote by $\pi$ the covering $ \cX_{r} \rightarrow \cX_{r}/C_{(2g+1)} \cong \mathbb{P}^1$. Then $\bar{i}$ fixes two points on $\cX_{r}/C_{(2g+1)}$, namely $\pi(P_\infty)$ and a point $R$ with $|\pi^{-1}(R)| = 2g+1$. Since $i$ commutes with $G$, it must fix $\pi^{-1}(R)$ pointwise, and the claim follows. 

If $p =2$, then $C_{(2g+1)}$ must normalize $i$. Let $\tilde{\cX}=\cX/\langle i \rangle $ and $\tilde{g}=g(\tilde{\cX})$. Then $\tilde{g} \leq 1$ since $\tilde{g} < g$. 
Suppose $\tilde{g}=1$. Then $C_{(2g+1)}$ must be isomorphic to a subgroup of a non-trivial one-point stabilizer in the automorphism group of an elliptic curve, a contradiction to \cite[Theorem 11.94]{hirschfeld-korchmaros-torres2008}, whence $\tilde{g} =0$ follows. 
\end{proof}

Henceforth, we shall assume that $\cX_r$ is not hyperelliptic. 

\begin{proposition}\label{tamenontame1}
Let $\cX_r$ be non-hyperelliptic 
 and let $P \in \Omega$. Then $\aut(\cX_r)_P^{(1)}$ is trivial. 
\end{proposition}
\begin{proof} Let us recall that $p \leq g+1$. By contradiction, let $\aut(\cX_r)_P^{(1)} \rtimes C_{(2g+1)} =  \aut(\cX_r)_P^{(0)} $, with $|\aut(\cX_r)_{P}^{(1)}| = p^s$ for $s \geq 1$. As $C_{2g+1}$ cannot be normal in $H$ by Remark \ref{remark:s3} and Lemma \ref{lem:2g+1hyp}, the number  $n_{2g+1}$ of Sylow $(2g+1)$-subgroups  in $\aut(\cX_r)_P^{(0)}$ is such that $n_{2g+1} \geq 2g+2$. This implies $|\aut(\cX_r)_{P}^{(1)}|> 2g+1$. Then by \cite[Theorem 11.78]{hirschfeld-korchmaros-torres2008}, $\bar{\cX_r} = \cX_r/\aut(\cX_r)_P^{(1)}$ is rational and $\{P\}$ is the unique short orbit of $\aut(\cX_r)_P^{(1)}$. Denote by $\pi$ the covering $\cX_r \rightarrow \bar{\cX_r}$. Then $\bar{C}_{2g+1}=  \aut(\cX_r)_P^{(0)}  /\aut(\cX_r)_P^{(1)}$ is a tame cyclic subgroup of $\PGL(2,K)$, whence it fixes two points in $\bar{\cX_r}$, namely $\pi(P)$ and a point $Q$ with $|\pi^{-1}(Q) | = p^s$. Then $\Omega \setminus\{P\}$ is contained in $\pi^{-1}(Q)$, whence $p^s = c(2g+1)+2$ for $c > 0$. Also, $n_{2g+1} = k(2g+1)+1 \geq 2g+2$ divides $p^s$, which implies $p^s = (c_1(2g+1)+2)(k(2g+1)+1) \geq (c_1(2g+1)+2)(2g+2)$ for an integer $c_1 > 0$. in fact, $c_1 =0 $ would yield $p =2$, and by applying the proof of Lemma \ref{lem:2g+1hyp} to a central involution of $\aut(\cX_r)_{P}^{(1)}$, we see that $\cX_r$ is hyperelliptic, a contradiction. Then 
$$
|\aut(\cX_r)_{P}^{(1)}| \geq (2g+3)(2g+2) > 4\frac{p}{(p-1)^2}g^2,
$$
a contradiction to \cite[Theorem 11.78 (iii)]{hirschfeld-korchmaros-torres2008}.

\end{proof}

\begin{lemma}\label{key}
Let $N$ be a non-trivial normal subgroup of $\aut(\cX_r)$. Then $C_{(2g+1)} \leq N$. 
\end{lemma}
\begin{proof} By contradiction, assume $C_{(2g+1)} \cap N = \{1\}$. Let $\bar{\cX}_r = \cX_r/C_{(2g+1)}$; also, let $\bar{g} = g(\bar{\cX}_r)$ and $\pi: \cX_r \rightarrow \bar{\cX}_r$ the corresponding Galois covering. Then $\bar{g} \leq 1$ since $\bar{g} < g$. 

Suppose $\bar{g}=1$. Then $\bar{C}= C_{(2g+1)}N/N$ must be isomorphic to a subgroup of a non-trivial one-point stabilizer in the automorphism group of an elliptic curve, a contradiction to \cite[Theorem 11.94]{hirschfeld-korchmaros-torres2008}.

If $\bar{g} = 0$, then $\bar{C}_{(2g+1)} = C_{(2g+1)}N/N$ is a cyclic subgroup of $\PGL(2,K)$ and, as such, it fixes two points $\bar{Q}, \bar{R}$ on $\bar{\cX}_r \simeq \mathbb{P}^1(K)$.  By Proposition \ref{tamenontame1} and Lemma \ref{lem:2g+1hyp}, $|\pi^{-1}(\bar{Q})| = |\pi^{-1}(\bar{R})| = |N|$. Then $\Omega = \{P_0,P_1, P_\infty\} \subset \pi^{-1}(\bar{Q}) \cup \pi^{-1}(\bar{R})$. This gives a contradiction, as we would either have $|N| \equiv 3 \mod 2g+1$ and  $|N| \equiv 0 \mod 2g+1$, or  $|N| \equiv 2 \mod 2g+1$ and $|N|=1 \mod 2g+1$. Then, our result follows.

\end{proof}

\begin{rem}
It can be checked that $C_{(2g+1)}$ is normal in $\aut(\cX_r)$ whenever the Hurwitz bound holds. In fact, Proposition \ref{tamenontame1} allows to rewrite the original proof of Theorem \ref{2g+1:facts} for $p >0$ and the hypothesis that $|\aut(\cX_r)| \leq 84(g-1)$. The proof and the computations involved are pretty similar to the ones we provide in Theorem \ref{teoremchar0g+1}. We briefly sketch the reasoning. First, by applying the Riemann-Hurwitz formula, one can prove that if $C_{(2g+1)}$ is self-normalizing in $\aut(\cX_r)$, then $C_{(2g+1)} = \aut(\cX_r)$. Since we are considering the case of non-hyperelliptic $(2g+1)$-curves, we are left to consider the case when $C_{(2g+1)}$ is normalized by a cyclic group of order $3$. If $C_{(2g+1)}$ is not normal in $\aut(\cX_r)$, then by the Sylow Theorem $|\aut(\cX_r)| \geq 3(2g+1)(2g+2) > 84(g-1)$ whenever $g >3$. 
\end{rem}

Hence, we may focus on the case when $|\aut(\cX_r)| > 84(g-1)$. An immediate consequence of Proposition \ref{tamenontame1} is the following. 
\begin{corollary}\label{biggroup1}
If $|\aut(\cX_r)| > 84(g-1)$, then case (d) of Theorem \ref{th:largeaut} holds.
\end{corollary}

\begin{lemma}\label{lematecnico1}
Let $\cX_r$ be  a non-hyperelliptic $(2g+1)$-curve of genus $g>2$. Suppose that $|\aut(\cX_r)| > 84(g-1)$. Then the following hold:
\begin{itemize}
\item[(i)] There is a cyclic automorphism $\tau \in \aut(\cX_r)$ of order $3$ normalizing $\alpha$ and acting transitively on $\Omega=\{P_0,P_1,P_\infty\}$.
\item[(ii)] $\aut(\cX_r)_{P_i}=C_{2g+1}$ for all $i=0,1,\infty$.
\item[(iii)] $|\aut(\cX_r)|=3(2g+1)(d(2g+1)+1)$ for some positive integer $d$.
\end{itemize}
\end{lemma}
\begin{proof}
If $|\aut(\cX_r)| > 84(g-1)$, then by Corollary \ref{biggroup1}, $\aut(\cX_r)$ has two short orbits on $\cX_r$: one tame, say $\Lambda_1$, and one non-tame, say $\Lambda_2$. Moreover, Proposition \ref{tamenontame1} implies $\Omega=\{P_0,P_1,P_\infty\} \subset \Lambda_1$. In particular, there  exists $\tau \in \aut(\cX_r)$ such that $\tau(P_0)=P_1$. Thus $\tau^{-1}\alpha \tau(P_0)=\tau^{-1}(P_1)=P_0$. Hence $\tau^{-1} \alpha \tau \in \aut(\cX_r)_{P_0}$, and then $\tau^{-1} \alpha \tau=\alpha$. Thus the order of $\tau$ equals $3$.  This proves (i). 

To prove (ii), assume that $|\aut(\cX_r)_{P_i}| > 2g+1$. By Proposition \ref{tamenontame1}, $\aut(\cX_r)_{P_i}$ is tame, whence cyclic. By \cite[Theorem 11.79]{hirschfeld-korchmaros-torres2008} we have $|\aut(\cX_r)_{P_i}|=2(2g+1)$, and $\cX_r$  is hyperelliptic by Lemma \ref{lem:2g+1hyp}, a contradiction. 

Finally, since $\Omega \subset \Lambda_1$, then (i) and (ii) imply that $3$ divides $|\Lambda_1|$ and $|\Lambda_1|\equiv 3 \mod 2g+1$. Thus,  $|\Lambda_1|=3d(2g+1)+3$ for some positive integer $d$. This combined with the Orbit-Stabilizer Theorem finishes the proof.

\end{proof}


Hence, by Lemma \ref{lematecnico1},  we may  assume $\cX_r$ non-hyperelliptic with   $r$ such that $r^2+r+1 \equiv 0 \mod (2g+1)$. Let $H = N_{\aut(\cX_r)}(C_{(2g+1)})$; then $H$ is a non-abelian group of order $3(2g+1)$. Note that this is possible only if $3 \mid g$ (as $3$ must divide $2g$). 

A key consequence of Lemma \ref{key} is that the intersection $\bar{N}$ of all non-trivial normal subgroups of $\aut(\cX)$ is non-trivial as it must contain $C_{(2g+1)}$; further, $\bar{N}$ must be bigger than $C_{(2g+1)}$ as $C_{(2g+1)}$ is not normal in $\aut(\cX_r)$ under our hypotheses.  Note that $\bar{N}$ is a minimal normal subgroup of $\aut(\cX_r)$ and, as such, it is characteristically simple (see \cite[p. 87]{robinson}). Then $\bar{N} = S^{(n)}$ for a finite simple group $S$ (\cite[p. 88]{robinson}). Actually, $n = 1$ by Lemma \ref{sylow}, whence $\bar{N} = S$. 

\begin{lemma}\label{minimalnormal}
 $|\bar{N}| > 84(g-1)$. 
\end{lemma}
\begin{proof} As $\bar{N}$ is simple, and $C_{(2g+1)}$ is s Sylow $(2g+1)$-subgroup of $\bar{N}$, we have $|\bar{N}| > (2g+1)(2g+2)$ as a group of order $q(q+1)$ for  a prime $q$, is non-simple. If  the  number  $n_{2g+1}$ of Sylow $(2g+1)$-subgroups is such that $n_{2g+1} \geq (6g+4)$, then the claim holds for $g > 4$, and we are done. If $n_{2g+1} = 4g+3$,then $$
|\bar{N}| \geq 4(2g+1)(4g+3) > 84(g-1)
$$ 

for every $g$, since $4$ must divide $|\bar{N}|$ ($\bar{N}$ simple) and $4g+3 \equiv 3 \mod 4$. Next, assume $n_{2g+1} = 2g+2$; then 
$$
|\bar{N}| \geq 2(2g+1)(2g+2) > 84(g-1)
$$
whenever $g > 7$.  We then deal with the cases $g =5, 6$ separately. If $g =5$, we must rule out the possibility that $|\bar{N}| = 264$, which is immediately disposed of, as there is no simple group of such order. If $g =6$, we must rule out the possibility that $|\bar{N}| = 364$, for which the same argument as in the former case applies.

\end{proof}

\begin{corollary}
If $|\aut(\cX_r)| > 84(g-1)$, then $\aut(\cX_r)$ is a finite simple group. 
\end{corollary}
\begin{proof} 
From our previous results, it follows that if $\aut(\cX_r)$ is non-simple, then it contains a simple minimal normal subgroup $\bar{N}$. Also,  a consequence of Lemma \ref{minimalnormal} is that the whole $H$ is a subgroup of $\bar{N}$. But then $\bar{N}$ contains the normalizer of a non-trivial Sylow subgroup of $\aut(\cX_r)$, and as a consequence of the Sylow Theorem, $\bar{N}$ is self-normalizing. In particular, it cannot be a non-trivial normal subgroup of $\aut(\cX_r)$, a contradiction. 

\end{proof}

\begin{theorem}\label{quasela}
If $|\aut(\cX_r)| > 84(g-1)$, then $|\aut(\cX_r)|  > 8g^3$. In this case, $\aut(\cX) \cong \PSU(3,p^h)$, and $\cX_r$ is isomorphic to a Hermitian curve. 
\end{theorem}
\begin{proof} 
By \cite[Theorem A]{vdovin}, if $G \not \cong \PSL(2,q)$ is a non-abelian finite simple group  and $A\leq G$ is abelian, then $|A|^3 < |G|$. Let $\aut(\cX_r) \not \cong \PSL(2,q) $.  We then have $|\aut(\cX_r)| > (2g+1)^3 > 8g^3$.  Since $\cX_r$ is not hyperelliptic, then by Theorem \ref{henn} either $p=2$, $\cX_{r}$ is isomorphic to a Suzuki curve, and $\aut(\cX_r)$ is isomorphic to a Suzuki group $Sz(2^{2h+1})$, for an integer $h >0$, or $\cX_{r}$ is isomorphic to a Hermitian curve and $\aut(\cX_r) \cong \PGU(3, p^k)$ for an integer $k > 0$. We can exclude the former possibility as $3 \nmid |Sz(2^{2h+1})|$.  Assume that the latter holds. As $\cX_r$ is a Hermitian curve, we have $2g+1 = q^2-q+1$ for a prime power $q \neq 2g+1$. Further, as $3 \mid g$, and $p\neq 3$, we have $q \equiv 1 \mod 3$. In particular, $d =1$ and $\aut(\cX_r) \cong \PSU(3,q)$ is simple. 

We need to exclude the possibility that $\aut(\cX_r) \cong \PSL(2,q)$. The subgroups of $\PSL(2,q)$ are known, see for instance \cite[Theorem A.8]{hirschfeld-korchmaros-torres2008}. Looking at this list, we see that either $q = 2g+1$ or $2g+1 \mid q\pm1$. If the former holds, there should  exist a cyclic group of order $g/2$ normalizing  $C_{(2g+1)}$, a contradiction to Lemma \ref{lematecnico1} (i) as $g > 3$. If the latter holds, the normalizer of $C_{(2g+1)}$ is a dihedral group, and again a contradiction is given by Lemma \ref{lematecnico1} (i) since the automorphism $\tau$ of order $3$ does not commute with  $C_{(2g+1)}$. 

\end{proof}

{\bf Proof of Theorem 3.5}. First, let $|\aut(\cX_r)| \leq 84(g-1)$. By Lemmas \ref{key} and \ref{minimalnormal}, $C_{(2g+1)}$ is a minimal normal subgroup of $\aut(\cX_r)$.  Then, $|\aut(\cX_r)| = c(2g+1)$, for $ c \in \{1,2,3\}$ by Remark \ref{remark:s3}. 

Next, assume that $|\aut(\cX_r)| >  84(g-1)$. By Theorem \ref{quasela}, we only need to see whether a Hermitian curve can be a $(2g+1)$-curve or not.  Recall that a Hermitian curve $\cH$ over $\mathbb{F}_{q^2}$ has affine equation $X^{q+1} = Y^q+Y$, is non-singular, thence it has genus $g(\cH) = (q^2-q)/2$. If $2g+1 = q^2-q+1$ is a prime, we must then look for cyclic subgroups of $\PSU(3,q)$ of this order normalized by an element of order $3$. By looking at the list of (maximal) subgroups of $\PSU(3,q)$ (see \cite[Theorem A.10]{hirschfeld-korchmaros-torres2008}), we see that there is one such subgroup, namely the normalizer of a Singer subgroup of order $q^2-q+1$, and the result follows.

\qed

\section{On $(g+1)$-curves}\label{g+1}

In this section, we classify $(g+1)$-curves up to birational equivalence, characterize hyperelliptic  $(g+1)$-curves and determine their automorphism groups. Henceforth, in order to simplify our notation, we let $C_{(g+1)} = G$ and  $ G = \langle \alpha \rangle$. Also, we denote by $\rho(\alpha)$ the number of points on $\cX$ that are fixed by $\alpha$.

\subsection{Classification}
\begin{lemma}\label{lem:p = g+1}
Let $\cX$ be a $(g+1)$-curve with $g >2$, and  let $\bar{g}$ be the genus  of the quotient curve $\cX/\langle \alpha \rangle$.
Then one of the following holds. 
\begin{enumerate}
\item[{\rm (i)}] If $ g+1 \neq p$, then $\bar{g} = 0$, and $\rho(\alpha) = 4$;
\item[{\rm (ii)}] If $g+1 = p$, then $\bar{g} = 0$, and $\rho(\alpha) \in \{1,2\}$. 
\end{enumerate}
\end{lemma}
\begin{proof}
As $g+1$ is prime, regardless of which case holds ($ p \neq g+1$ or $p = g+1 $), the Riemann-Hurwitz genus formula applied to the covering $\cX \rightarrow \cX/\langle\alpha\rangle$ reads 
\begin{equation}\label{eq:g+1}
2(g-1) = 2(g+1)(\bar{g}-1)+ sg. 
\end{equation}
for a non-negative integer $s$. From direct inspection, we see that $\bar{g} \leq 1$. Assume that $\bar{g} =1$; then Equation \eqref{eq:g+1} yields
$$
g = \frac{2}{2-s},
$$
a contradiction. Hence, $\bar{g} = 0$. 

If $g+1\neq p$, then $s = \rho(\alpha)$. In this case, Equation \eqref{eq:g+1} reads 
\begin{equation}\label{eq:4}
4g = sg,
\end{equation}
whence $s =4$, and item (i)  follows.

 If $g+1 = p$, Equality \eqref{eq:4} still holds true. However,  higher ramification groups at each point that is fixed by $\alpha$ have to be taken into account. Let $P$ be one of such points.  As $G_{P}^{(0)} = G_{P}^{(1)} = G$, we have that the different exponent $d_P$ of $P$ (see \cite[Theorem 11.70]{hirschfeld-korchmaros-torres2008}) is such that $d_P \geq 2g$, whence either $P$ is the only point that is fixed by $\alpha$, and $G_{P}^{(0)} = G_{P}^{(1)} = G_{P}^{(2)} = G_{P}^{(3)} = G$ while $G_{P}^{(4)}$ is trivial, or there exists a further fixed point $Q$ and $G_{P}^{(0)} = G_{P}^{(1)} = G_{Q}^{(0)} = G_{Q}^{(1)} = G$ with $G_{P}^{(2)},G_{Q}^{(2)}$ both trivial.  
 Hence, item (ii) follows. 
\end{proof}

\begin{theorem}\label{thm:p = g+1}
Let $\cX$ be a $(g+1)$-curve with $g >2$. Then one of the following holds.
\begin{enumerate}
\item[{\rm (a)}]
If $\cX$ is tame, then $\cX$ is  birationally equivalent to one of the following plane curves: $$\cX_{r,s,t,a}: Y^{g+1} = X^r(X-1)^s(X-a)^t,$$ where $a \in K\setminus\{0,1\}$ and  $r,s,t <g+1$, with $r+s+t \not\equiv 0\mod (g+1)$.  
\item[{\rm (b)}] 
If $\cX$ is wild, then  $\cX$ is  birationally equivalent either to a (hyperelliptic) curve 
$$
\cY_{a,b,c}: Y^{g+1}-Y = \frac{aX^2+bX+c}{X(X-1)}, 
$$
with $(a,b,c) \neq (0,0,0)$ and $\gcd(aX^2+bX+c, X(X-1)) = 1$, or to a curve 
$$ \cZ_{d,e,\ell}: Y^{g+1}-Y = X^3+dX^2+eX+\ell,$$
for $d,e,\ell \in K$. 
\end{enumerate}
\end{theorem}
\begin{proof} (a). By Lemma \ref{lem:p = g+1}, the cover $\cX \rightarrow  \cX/G \cong \mathbb{P}^1$ is a Kummer cover totally ramifying at $4$ points. As $\aut(\mathbb{P}^1) \cong \PGL(2,K)$ is sharply $3$-transitive on the points of $\mathbb{P}^1$, we may assume that $3$ of such points are $\bar{P}_{\infty} = \infty, \bar{P}_{0} = (0:1) $ and $\bar{P}_1 = (1:1)$, the fourth being given by $\bar{P}_a = (a:1)$ for some $a \neq 0,1$. This means $K(\cX) = K(x,y)$ with $y^{g+1} = h(x)$, for a polynomial $h(x) \in K(x)$ that has its zeroes at $ \bar{P}_{0},  \bar{P}_{1}$ and  $\bar{P}_a$, and whose degree is coprime with $g+1$. 
The result is then a consequence of \cite[Proposition 3.7.3]{stbook}

(b). By Lemma \ref{lem:p = g+1}, the cover $\cX \rightarrow  \cX/G \cong \mathbb{P}^1$ is an Artin-Schreier cover  totally ramifying either at one point $P$ or at two points $P,Q$.  In the former case, the function field of $\cX$ can be written in standard form as $K(x,y)$ with $y^{g+1}-y = f(x)$ for a function in $K(x)$ with exactly one pole of order $3$, see \cite[Proposition 3.7.8]{stbook}, whence $f$ is a degree-3 polynomial, that can be assumed to be monic since $K$ is algebraically closed. Hence, 
$$
f(x) = x^3+dx^2+ex+\ell,
$$
for $d,e,\ell \in K$.  

In the latter case, by  \cite[Proposition 3.7.8]{stbook} we have $K(\cX) = K(x,y)$ with $y^{g+1}-y  = g(x)$ for a function $g(x) \in K(x)$ with exactly two simple poles. We may choose such poles to be $(0:1), (1:1)$; further, the numerator of $g$ is coprime with $x(x-1)$ and has degree less than or equal to $2$. Summing up these informations, we get   
$$
g(x) = \frac{ax^2+bx+c}{x(x-1)}, 
$$
with $(a,b,c) \neq (0,0,0)$ and $\gcd(ax^2+bx+c, x(x-1)) = 1$, and our claim follows. 
\end{proof}

\begin{rem}
The full list of $(g+1)$-curves for $g = 2$, as well as their automorphism groups, can be found in \cite{shaska}. 
\end{rem}

\subsection{Hyperelliptic $(g+1)$-curves}

Throughout this subsection, $\cX$ is a $(g+1)$-curve of genus $g > 2$. Our goal here is to characterize all hyperelliptic $(g+1)$-curves when $p \neq 2$. As we will see, the most difficult (and most interesting) case is when $\cX$ is tame. 

\begin{theorem}\label{thm:hypaut} The following hold.
\begin{itemize}
\item If $\cX$ is wild, then $\cX$ is hyperelliptic if, and only if, it is birationally equivalent to the curve $\cY_{a,b,c}$ as in Theorem \ref{thm:p = g+1}(b). 
\item If $\cX$ is tame and $p >2$, then $\cX$ is hyperelliptic if, and only if, it is birationally equivalent to the plane curve
$$
Y^2 = (X^{g+1}-a_o)(X^{g+1}-a_1),
$$
with $a_0,a_1 \in K^*$. 
\end{itemize}
\end{theorem}

\begin{proof}
If $\cX$ is wild, our result is a straightforward corollary of Theorem \ref{thm:p = g+1}(b), as it is clear from its equation that the curve $\cY_{a,b,c}$ is hyperelliptic. It is then enough to check that curves 
$$ \cZ_{d,e,\ell}: Y^{g+1}-Y = X^3+dX^2+eX+\ell$$

are not hyperelliptic. By contradiction, let  $\cZ_{d,e,\ell}$ be hyperelliptic. Let $K(\cZ_{d,e,\ell}) = K(x,y)$; then the pole divisor  of $y$ is given by $(y)_{\infty} = 3P_\infty$, where $P_\infty$ is the unique place of $K(\cX)$ centered at $(1:0:0)$. Since the gap sequence of a hyperelliptic curve over $K$ is classical (see e.g. \cite[Satz 8]{schm}) and $g \geq 4$, then $P_\infty$ is a Weierstrass point for $\cZ_{d,e,\ell}$. Let $H(P_\infty)$ be the Weierstrass semigroup at $P_\infty$. Then $3 \in H(P_\infty)$. But $2 \in H(P_\infty)$ as $P_\infty$ is a Weierstrass point, which gives that $H(P_\infty)=\{0,2,3,\ldots\}$, a contradiction since $g >1$.

Let $p\neq g+1$ with $p >2$. Let $i$ denote the hyperelliptic involution of $\cX$, and $\alpha \in \aut(\cX)$ with $\alpha^{g+1} = 1$. Then $\alpha \in \aut(\cX)/\langle i \rangle $ as $g+1$ is odd. By \cite[Satz 5.1 and 5.6, Lemma 5.5]{brandt}, $\aut(\cX)$ has a cyclic automorphism group of order $g+1$ if and only if there exist $u,t$ generators of $K(\cX)$ such that 
$$
u^2 = t^\nu\prod_{j=0}^{s-1}(t^{g+1}-a_j),
$$
where $\nu \in \{0,1\}, s \in \mathbb{N}_+$ and the $a_j \in K^*$ are pairwise distinct. Thus $\nu = 0, s= 2$, and our claim follows. 
\end{proof}

\begin{proposition}
Let $\cX_{r,s,t,a}: Y^{g+1} = X^r(X-1)^s(X-a)^t$ be a tame $(g+1)$-curve. If $r+s+t \leq g$, then $\cX_{r,s,t,a}$ is not hyperelliptic. 
\end{proposition}
\begin{proof} By contradiction, assume that $\cX_{r,s,t,a}$ with $r+s+t \leq g$  is hyperelliptic. Let $P_\infty$ be the unique place centered at $(1:0:0)$. On the one hand, by Theorem \ref{thm:hypaut} (b), we have that $P_\infty$ cannot be a Weierstrass point for $\cX_{r,s,t,a}$ as it is fixed by the automorphism  of order $(g+1)$ given by $\alpha(x,y) = (x,\lambda y)$, where $\lambda$ is a primitive $(g+1)$-th root of the unity. 
 Hence, $P_\infty$ has the classical gap sequence $G(P_\infty) = \{1,\ldots,g\}$ as $\cX_{r,s,t,a}$ is hyperelliptic. On the other hand, as $(y)_\infty = (r+s+t)P_\infty$, we have that $l= r+s+t \leq g$ is a non-gap at $P_\infty$, a contradiction. 
\end{proof}

We are now in a position to characterize all tame hyperelliptic $(g+1)$-curves for $p \neq 2$.

\begin{proposition}\label{carattip}
Let $\cX_{r,s,t,a}$ be a tame $(g+1)$-curve defined by
$$
Y^{g+1} = X^r(X-1)^s(X-a)^t,
$$
where $a \in K \backslash\{0,1\}$, $r,s,t<g+1$, $r+s+t \not\equiv 0 \mod g+1$ and $(r,s,t) \neq w(r',s',t')$ where $r'+s'+t' \leq g$. Then $\cX_{r,s,t,a}$ is hyperelliptic if, and only if, $\{r,s,t\}=\{d,g+1-d\}$ with $0<d<g+1$. 
\end{proposition}
\begin{proof}
Assume that $r=s$ and $r+t=g+1$. As before, set $G=\langle \alpha \rangle$, where $\alpha:(x,y)\mapsto (x,\lambda y)$ and $\lambda$ is a primitive $(g+1)$-th root of unity.  Denote by $P_0,P_1,P_a$ and $P_\infty$ the points of $\cX_{r,s,t,a}$ lying over the points $\tilde{P}_0=(0:1), \tilde{P}_1=(1:1), \tilde{P}_a =(a:1)$ and $\tilde{P}_\infty=(1:0) \in \cX_{r,s,t,a}/G$ respectively. Since $r=s$, we have 
$$
(y)=rP_0+rP_1+tP_a-(2r+t)P_\infty \ \ \ \text{ and } \ \ \ (x-i)=(g+1)P_i-(g+1)P_\infty,
$$
for $i=0,1,a$. Since $g+1$ is prime and $r<g+1$, there exist $m,n \in \mathbb{Z}$ such that $-mr+n(g+1)=1$. For $\ell=m-n$, define
$$
f=\frac{y^m}{x^n(x-1)^n(x-a)^\ell} \in K(\cX_{r,s,t,a}).
$$
Then
\begin{eqnarray}
(f) &=&(mr-n(g+1))(P_0+P_1)+(mt-\ell(g+1))P_a+((2n+\ell)(g+1)-m(2r+t))P_\infty \nonumber \\
    &=& P_\infty+P_a-P_0-P_1. \nonumber
\end{eqnarray}
In particular, $[K(\cX_{r,s,t,a}):K(f)]=2$, and so $\cX_{r,s,t,a}$ is hyperelliptic.

Conversely, if $\cX_{r,s,t,a}$ is hyperelliptic, then there exists an involution $\mu \in \aut(\cX_{r,s,t,a})$ which fixes a rational subfield. Also, $\mu$ is central in $\aut(\cX_{r,s,t,a})$, whence it commutes with $G$.
 In particular, $\mu$ must act semi-regularly on the set $\Omega = \{P_0,P_1,P_a,P_\infty\}$. Without loss of generality, we may assume that $\mu(P_0) = \mu(P_1), \mu(P_a) = P_\infty$. As  $(x-a)=(g+1)P_a-(g+1)P_\infty$, there is a constant $c \in K^{*}$ such that

$$
\mu(x)=\frac{x-\frac{a^2-c}{a}}{x/a-1}.
$$
Since the only zero of $\mu(x)$ is $P_1$, we conclude that $\frac{a^2-c}{a}=1$, i.e., $c=a^2-a$. Let $f_k \in K(x)$, with $k=0,\ldots,g$ such that

\begin{equation}\label{comb}
\mu(y)=\sum_{k=0}^{g}f_ky^k.
\end{equation}
Combining \eqref{comb} with the facts that $\mu \alpha=\alpha \mu$ and $\alpha(f_k)= f_k$, we obtain
$$
\sum_{k=0}^{g}(\lambda^k-\lambda)f_ky^k=0,
$$
which gives that $f_k=0$ for $k \neq 1$, as the set $\{f_0,f_1 y,\ldots,f_gy^g\}$ is linearly independent over $K$. Therefore, there is $f(x) \in K(x)$ such that $\mu(y)=f(x)y$. Now, on the one hand, the expression of $\mu(x)$ provides
$$
\mu(y^{g+1})=\mu(x^r(x-1)^s(x-a)^t) = a^{r+t}(a-1)^{s+t} \frac{x^s(x-1)^r}{(x-a)^{r+s+t}}.
$$
On the other hand, from $\mu(y)=f(x)y$ we obtain
$$
\mu(y^{g+1})= f(x)^{g+1}x^r(x-1)^s(x-a)^t.
$$ 
Comparing both equations, we obtain 
$$
f(x)^{g+1}=  \frac{a^{r+t}(a-1)^{s+t}x^{s-r}(x-1)^{r-s}}{(x-a)^{r+s+2t}}.
$$ 
Thus $g+1$ divides both $r-s$ and $r+s+2t$. Since $r,s,t \leq g$, we conclude that $r=s$ and $r+t=g+1$.  In particular, $\mu(y)=\frac{a(a-1)y}{(x-a)^2}$.
\end{proof}

\subsection{Automorphism groups of tame $(g+1)$-curves}

Throughout this subsection, $\cX$ is a tame, non-hyperelliptic $(g+1)$-curve. Recall that $G \leq \aut(\cX)$ is a group of prime order equal to $(g+1)$, $\alpha$ is a generator of $G$ and  $\Omega = \{P_0,P_1,P_a,P_\infty\}$ is the set of points fixed by $\alpha$. We compute the full automorphism group $\aut(\cX)$ through a number of (mostly) group-theoretical results. More in detail, we aim to prove the following result. 

\begin{theorem}\label{resumog+1}
For $g \geq 4$, let $\cX$ be a tame non-hyperelliptic $(g+1)$-curve defined over an algebraically closed field $K$ of characteristic $p \neq 2,3$. Then $G$ is normal in $\aut(\cX)$, whence $|\aut(\cX)| =c(g+1)$, for $c \in \{1,2,3,4\}$, unless $g = 4$ and $\cX$ is isomorphic to the curve 
$$
\mathcal{Q}: Y^5 = X(X-1)^4(X-\ha)^2. 
$$
\end{theorem}

\begin{lemma}
$G$ is a Sylow $(g+1)$-subgroup of $\aut(\cX)$. 
\end{lemma}
\begin{proof}  By contradiction, let $S$ be a Sylow $(g+1)$-subgroup of $\aut(\cX)$ with $|S| = (g+1)^i$, $ i \geq 2$. Then $S$ has a subgroup $S_1$ of order $(g+1)^2$. Then $S_1$ is abelian, a contradiction to  \cite[Theorem 11.79]{hirschfeld-korchmaros-torres2008}. 
\end{proof}

\begin{rem}\label{remark:s4}
Let $N$ be the normalizer of $G$ in $\aut(\cX)$. Then  $N$ must act on the set $\Omega $, that is, there exists a permutation representation $\rho: N \rightarrow {\rm S}_4$. A computation via the Riemann-Hurwitz formula shows that $\Ker(\rho)  = G$. Hence, $N/G \leq {\rm S}_4 \leq \PGL(2,K)$.  

\end{rem}

\begin{proposition}\label{prop:c3}
Let $C_3 \leq \aut(\cX)$ with $ |C_3| = 3$. If $C_3$ normalizes $G$, then either 
\begin{enumerate}
\item[{\rm(a)}] $p \neq 3$, and $\cX$ is birationally equivalent to the curve $Y^{g+1}+X^3+1 = 0$;
\item[{\rm (b)}] $p = 3$, and $\cX$ is birationally equivalent to the curve $Y^{g+1} = X^3-X$.
 \end{enumerate}
\end{proposition}
\begin{proof}Let $C_3 = \langle \beta \rangle $. First, we prove that $H = \langle \alpha,\beta \rangle \simeq C_3 \times G$. As $C_3$ must act on $\Omega$ and $\beta$ cannot fix $\Omega$ pointwise,  then $\beta$ must fix a point in $\Omega$, say $P_0$, and act semi-regularly on $\Omega \setminus \{P_0\}$. If $p \neq 3$, this implies that $H \cong C_{3(g+1)} \cong C_3 \times G$. If $p = 3$, then $C_3$ has to be normal in $H$, and again $H \cong C_{3(g+1)} \cong C_3 \times G$. 

We now prove that $\bar{g} = g(\cX/C_3) = 0$. By contradiction, let $\bar{g} \geq 1$. 

If $\bar{g} = 1$, then $G$ must be isomorphic to a subgroup of a non-trivial one-point stabilizer in the automorphism group of an elliptic curve, a contradiction to \cite[Theorem 11.94]{hirschfeld-korchmaros-torres2008}. 

Let $\bar{g} \geq 2$. Then by Theorem \ref{homma1}, 
$g+1 = 2\bar{g}+1$. By applying the Riemann-Hurwitz formula to the cover $\pi: \cX\rightarrow \cX/C_3$, we get 
$$
4\bar{g}-2 \geq 6(\bar{g}-1)+2s,
$$

for a positive integer $s$, that is, 

$$
\bar{g} \leq 2-s. 
$$

However, we have $s \geq 1$ as $\beta(P_0) = P_0$. This is a contradiction, whence $\bar{g} = 0$. 

If $p \neq 3$, as $\cX/G \cong \mathbb{P}^1$, $\bar{C}_3 \cong H/G$ is a  tame cyclic subgroup of $\PGL(2,K)$, and hence  it fixes two points on $\mathbb{P}^1$, $\bar{P}_0 = \pi(P_0)$ and $\bar{Q} \in \mathbb{P}^1\setminus \pi(\Omega)$. This implies that $\langle \alpha, \beta  \rangle $ has three short orbits on $\cX$, namely $\{P_0\}, \Omega \setminus \{P_0\}$ and $\Omega_1$ with $|\Omega_1| = g+1$. Item $(a)$ then follows from \cite[Theorem 5.1 (a)]{AS}. 

If $p = 3$, then $\bar{C}_3 \cong H/G$ is a  non-tame cyclic subgroup of $\PGL(2,K)$, and hence  it fixes exactly one point on $\mathbb{P}^1$, which is $\bar{P}_0 = \pi(P_0)$. Thus the cover $\cX \rightarrow \cX/C_3$ is an Artin-Schreier cover of $\mathbb{P}^1$, and  $\cX/(C_3 \times G)$ has only one fully ramified point. Therefore, there exist $x,y \in K(\cX)$ such that $K(\cX/C_3)=K(y)$, $K(\cX/G)=K(x)$, and $K(\cX/(C_3\times G))=K(x^3-x)=K(y^{g+1})$ with $x^3-x$ and $y^{g+1}$ having a common pole. Hence, $x^3-x=ay^{g+1}+b$, for some $a,b \in K^{*}$. From Galois Theory, one has that $K(\cX)=K(x,y)$, and item (b) follows from a change of coordinates of $(x,y)$. 

\end{proof}

\begin{proposition}\label{prop:involution}
Let $ p \neq 2$, and let $i \in \aut(\cX)$ be an involution normalizing $\alpha$. Let $H =  \langle i \rangle \ltimes G$.  Then, one of the following occurs: 
\begin{enumerate}
\item[{\rm (1)}] H is cyclic, and $i $ has exactly two fixed points in $\Omega$;
\item[{\rm (2)}] H is cyclic, and $i$ fixes no point in $\Omega$; in this case, $\cX$ is hyperelliptic; 
\item[{\rm (3)}] H is dihedral, and $i$ has exactly two fixed points in $ \cX\setminus \Omega$. 
\end{enumerate}
\end{proposition}
\begin{proof} It is immediately seen  that $H$ is either cyclic or dihedral. Further, as $i $ normalizes $\alpha$, we have that $\bar{i} = H/\langle \alpha \rangle$ is an involution in $\PGL(2,K)$. As $p >2$, then $i$ fixes two points $\bar{Q} ,\bar{R} \in \mathbb{P}^1$. Using the same notation as in Proposition \ref{prop:c3}, we see that either $\bar{Q} ,\bar{R} \in \pi(\Omega)$ or $\bar{Q} ,\bar{R} \not\in \pi(\Omega)$. 

If the former holds, then $H$ is cyclic and $i$ has exactly two fixed points. Without loss of generality, we may assume that such points are $P_0$ and $P_\infty$. Also, by \cite[Proposition 3.7.3]{stbook}, $\cX$ is birationally equivalent to a curve with affine equation
$$
\cH: Y^{2(g+1)} = X^r(X-1)^{2s},
$$
with $\gcd(r,2(g+1)) = \gcd(s,g+1) = \gcd(r+2s,2(g+1)) = 1$. 

If the latter holds, then either $i$ commutes with $\alpha$ or $i\alpha = \alpha^{-1}i$. 
If $i\alpha = \alpha i $, then again $H$ is cyclic. This time, however, $\langle i,\alpha\rangle$ has four short orbits, namely $\Omega_1$ and $\Omega_2$, (each of size two), with $\Omega_1 \cup \Omega_2 = \Omega$ and $\Omega_3$ and $\Omega_4$ (each of size $g+1$), with  $\Omega_3 = \pi^{-1}(\bar{Q})$ and $\Omega_4 = \pi^{-1}(\bar{R})$. Clearly, $i$ must fix a point $Q_1 \in \Omega_3$ and a point $R_1 \in \Omega_4$. But $i$ is the only involution in the cyclic group $H$, whence it fixes $\Omega_3$ and $\Omega_4$ pointwise. Then, $\cX$ is hyperelliptic as $i$ fixes $2g+2$ points. 

We are left with the case when $H$ is dihedral. As in the previous case, $H$ has four short orbits, namely $\Omega_1$ and $\Omega_2$ of size two, with $\Omega_1 \cup \Omega_2 = \Omega$ and $\Omega_3$ and $\Omega_4$ of size $g+1$, with  $\Omega_3 = \pi^{-1}(\bar{Q})$ and $\Omega_4 = \pi^{-1}(\bar{R})$, and $i$ fixes at least a point $Q_1 \in \Omega_3$ and a point $R_1 \in \Omega_4$. We claim that $Q_1, R_1$ are the only points that are fixed by $i$. By \cite[Application 3, $n$ odd]{Ac}, one has $g(\cX/\langle i \rangle ) = g/2$, and the assertion follows via the Riemann-Hurwitz formula.

\end{proof}

\begin{rem}
From the proof of \ref{prop:involution} (2), it is easily seen that when $p \neq 2$, a tame $(g+1)$-curve is hyperelliptic if, and only if, it is birationally equivalent to a generalized Fermat curve, see \cite{AS}. Although results on such curves where stated (and proved) over an algebraic closure of a finite field, all the results regarding the automorphism groups of generalized Fermat curves are still valid over any algebraically closed field. Thence, the full automorphism group of a tame, hyperelliptic $(g+1)$-curve can be found in \cite[Theorem 6.11]{AS}. 
\end{rem}

\begin{proposition}\label{prop:V4}
Let $p \neq 2$, and let $H/G$ be a subgroup of $N/G$ such that $|H/G| = 4$. If $\cX$ is not hyperelliptic, then $H/G$ cannot be isomorphic to the Klein Viergruppe $V_4$.
\end{proposition}
\begin{proof} By contradiction, let $H/G \cong V_4$. Then either $H \cong V_4 \times G$ or $H \cong D_{2(g+1)}$, where $D_{2(g+1)}$ is a dihedral group of order $4(g+1)$. If the former holds, each one of the three involutions in $V_4$ commutes with $G$, and at least one of them must fix a point outside $\Omega$. By Proposition \ref{prop:involution} (2), $\cX$ is hyperelliptic, a contradiction. 

If the latter holds, then exactly one involution $i_1$ in $V_4$ commutes with $G$, while the other  two involutions $i_2, i_3$ in $V_4$ give rise to two conjugacy classes of involutions in $D_{2(g+1)}$ which cannot commute with $G$, and hence, fix exactly two points each outside $\Omega$. Say that $i_1$ fixes $P_0$ and $P_\infty$. Then $i_2(P_0) = P_\infty$ and $i_3(P_0) = P_\infty$ as $V_4$ is abelian. This implies that (without loss of generality) $i_2(P_1)=P_a$ and $i_3(P_1)=P_a$. Thus, $i_1(P_1)=(i_2i_3)(P_1)=P_1$, contradicting the fact that $i_1$ fixes only two points.

\end{proof}

\begin{corollary}\label{cor:4.14}
If $\cX$ is not hyperelliptic, then $|N/G| \leq 6$, with equality holding if and only if $p =3$, and $\cX$ is isomorphic to the curve $\cG$ with affine equation 
$$
\cG: X^{g+1} = Y^3-Y.
$$
\end{corollary}
\begin{proof} By Proposition \ref{prop:V4}, we immediately get $|N/G| \leq 6$. Let $p \neq 3$. If there exists $\beta \in N$ of order $3$, then by Proposition \ref{prop:c3} (a) $\cX$ is birationally equivalent to a generalized Fermat curve with affine equation
$$
\cF: Y^{g+1}+X^3+1 = 0.
$$
By \cite[Theorem 1]{Ko}, one has $\aut(\cF) \cong G \times \langle \beta \rangle $ as $ 3 \nmid (g+1)$. 

If $p = 3$, by Proposition \ref{prop:c3} (b), $\cX$ is birationally equivalent to the curve

$$
\cG: Y^{g+1} = X^3-X. 
$$
Let $K(\cG) = K(x,y)$. The map $\gamma(x,y) = (-x,-y)$ is an involution in $\aut(K(\cG))$ normalizing $G$, and the assertion follows. 

\end{proof}

\begin{proposition}\label{4malizer}
If $p \neq 2$, assume that there exists an automorphism $\delta \in \aut(\cX)$ of order $4$ normalizing $G$. Then, with notation of Theorem \ref{thm:p = g+1}, we have $r+s = g+1$, $r^2+t^2 \equiv 0 \mod g+1$, $a = 1/2$, $t \neq r,s$ and if $\delta^{-1}\alpha\delta = \alpha^j$, then $j^2+1 \equiv 0 \mod g+1$. 
\end{proposition}
\begin{proof}
First, note that $\delta$ acts transitively on $\Omega=\{P_0,P_1,P_a, P_\infty\}$. Indeed, since $\delta$ normalizes $G$, it has to act on the set of $4$ points below those of $\Omega$ in $\mathbb{P}^1$, and the number of fixed points of a cyclic group in $\mathbb{P}^1$ is $2$. Thus we may assume that $\delta(P_0)=P_a$, $\delta(P_a)=P_1$, $\delta(P_1)=P_\infty$ and $\delta(P_\infty)=P_0$. Set $\mu=\delta^2$. 
 Now $\mu$ normalizes $G =\langle \alpha \rangle$ but it does not commute with $\alpha$, since otherwise it would follow from the proof of  Proposition \ref{carattip} that $\cX$ is hyperelliptic. Thus $\mu \alpha \mu=\alpha^g$. In particular, if $\delta^{-1}\alpha\delta = \alpha^j$, we have $\alpha^{j^2+1}=1$, and then $g+1 | j^2+1$. Arguing as in Proposition \ref{carattip}, we obtain
$$
\mu(x)=\frac{a(x-1)}{x-a} \ \ \ \text{ and } \ \ \ \mu(y)=fy^g, \ \ \ f \in K(x).
$$
Hence
$$
f^{g+1}=\frac{a^{r+t}(a-1)^{s+t}x^{s-rg}(x-1)^{r-sg}}{(x-a)^{r+s+t(g+1)}},
$$
which gives $r+s=g+1$.

Now, $(x)=(g+1)P_0-(g+1)P_\infty$ and $\delta:(P_0,P_\infty)\mapsto (P_a,P_0)$ implies 
$$
(\delta(x))=\left(\frac{x-a}{x}\right),
$$
thus $\delta(x)=c(x-a)x^{-1}$ for some $c \in K$. Analogously, we obtain $\delta(x-a)=c_1(x-1)x^{-1}$ and $\delta(x-1)=c_2x^{-1}$, with $c_1,c_2 \in K$. Comparing these equations, we obtain $a=2^{-1}$ and $c=1$. Once again arguing as in Proposition \ref{carattip} gives $\delta(y)=wy^j$ for some $w \in K(x)$, and a computation leads to $r^2+t^2 \equiv 0 \mod g+1$ and  $t \neq r,s$. 
\end{proof}

\begin{proposition}\label{tamenontame}
For $p >3$, let $\cX$ be non-hyperelliptic and let $P \in \Omega$. Then $\aut(\cX)_P^{(1)}$ is trivial. 
\end{proposition}
\begin{proof} For $p > g+1$, the claim follows by Theorem \cite{Roq}. Let $p \leq g-1$. By contradiction, let $H= \aut(\cX)_P^{(1)} \rtimes G \leq \aut(\cX)_P^{(0)} $, with $|\aut(\cX)_{P}^{(1)}| = p^r$ for $r \geq 1$. Since $p >3$, $G$ cannot be normal in $H$ by Remark \ref{remark:s4}, whence the number  $n_{g+1}$ of Sylow $(g+1)$-subgroups  in $H$ is such that $n_{g+1} \geq g+2$. As $n_{g+1} \mid p^r$, we see that $n_{g+1} = p^s, s \leq r$, whence $n_{g+1}  = g+2 $ is not possible as $g+2$ is even. Then $n_{g+1} \geq 2(g+1)+1$, which in turn implies $|\aut(\cX)_{P}^{(1)}|> 2g+1$. Then by \cite[Theorem 11.78]{hirschfeld-korchmaros-torres2008}, $\bar{\cX} = \cX/\aut(\cX)_P^{(1)}$ is rational and $\{P\}$ is the unique short orbit of $\aut(\cX)_P^{(1)}$. Denote by $\pi$ the covering $\cX \rightarrow \bar{\cX}$. Then $\bar{G} = H/\aut(\cX)_P^{(1)}$ is a tame cyclic subgroup of $\PGL(2,K)$ fixing two points in $\bar{\cX}$, namely $\pi(P)$ and a point $Q$ with $|\pi^{-1}(Q) | = p^r$. Then $\Omega \setminus\{P\}$ is contained in $\pi^{-1}(Q)$, whence $p^r = c(g+1)+3$ for $c > 0$ (recall that $p >3$). Also, $n_{g+1} = k(g+1)+1 \geq 2(g+1)+1$ divides $p^r$, which implies $p^r = (c_1(g+1)+3)(k(g+1)+1) \geq (c_1(g+1)+3)(2(g+1)+1)$ for an integer $c_1 \geq 0$. Again, $c_1 > 0$ as $p >3$, and $c_1 \geq 2$ as $g+4$ is even. Then 
$$
|\aut(\cX)_{P}^{(1)}| \geq 4(g+1)^2 > 4\frac{p}{(p-1)^2}g^2,
$$
a contradiction to \cite[Theorem 11.78 (iii)]{hirschfeld-korchmaros-torres2008}.

\end{proof}

As an immediate yet crucial corollary of Proposition \ref{tamenontame}, we obtain the following. 
\begin{corollary}\label{biggroup}
Let $p > 3$. If $\aut(\cX)$ has less than $3$ short orbits, then case (d) of Theorem \ref{th:largeaut} holds.
\end{corollary}

\begin{theorem}\label{teoremchar0g+1}
For $p \neq 2,3$, let $\aut(\cX)$ have $k \geq 3$ short orbits. Then $G$ is normal in $\aut(\cX)$, unless $g =4$ and $\cX$ is isomorphic to the curve
$$
\mathcal{Q}: Y^5 = X(X-1)^4(X-\ha)^2, 
$$
with $\aut(\mathcal{Q}) \cong {\rm S}_5$. In particular, for $g >4$, one has $|\aut(\cX)| = c(g+1)$ for $c \in \{1,2,3,4\}$. 
\end{theorem}
\begin{proof} We will make frequently use of the following consequence of Proposition \ref{tamenontame}: the stabilizer of a point in $\Omega$ has size $b(g+1)$, for $b\in \{1,2,3\}$. 

By contradiction, let $G$ be non-normal in $\aut(\cX)$. Then $|\aut(\cX)| \geq (g+1)(g+2)$. Also, $G$ cannot be self-normalizing. To see this, suppose not. Then no two of the points in $\Omega$ can lie in the same $\aut(\cX)$-short orbit. Without loss of generality, assume that there exists $\tau \in \aut(\cX)$ such that $\tau(P_0)=P_1$. Thus, $\tau^{-1}\alpha \tau(P_0)=\tau^{-1}(P_1)=P_0$. Hence, $\tau^{-1} \alpha \tau \in \aut(\cX)_{P_0}$, and then $\tau^{-1} \alpha \tau=\alpha$, that is, $G$ is not self-normalizing. Now, let $k$ be the number of $\aut(\cX)$-short orbits; then  $k \geq 4$, and $|\aut(\cX)| \leq 12(g-1)$ by the proof of \cite[Theorem 11.56]{hirschfeld-korchmaros-torres2008}.  This gives a contradiction for $g >7$. We are then left with the cases $g = 4, |\aut(\cX)| = 30$ and $g = 6, |\aut(\cX)| = 56$. In both cases, $\aut(\cX)$ is a non-abelian group of order $q(q+1)$ where $q = (g+1)$ is a prime. Since $G$ is not normal in $\aut(\cX)$, then by the Sylow theorem we have that a group of order $g+2$ is normal in $\aut(\cX)$ and it is the unique  Sylow $2$-subgroup of $\aut(\cX)$ (that is, $g+1$ is a Mersenne prime). Then $g = 6$,  $k = 4$ (if $k \geq 5$, then $|\aut(\cX)| \leq 20$, a contradiction) and $|\aut(\cX)_{P_i}| = 7$.
A computation via the Riemann-Hurwitz formula gives a contradiction.  

This means that  $\cX$ must be (birationally equivalent to) one of the curves described in Propositions \ref{prop:c3}, \ref{prop:involution}, or \ref{4malizer}. 

The case of Proposition \ref{prop:c3} is immediately settled, as in this case, $\cX$ is isomorphic to the curve $\cF$ (since $p \neq 3$) and $|\aut(\cX)| = 3(g+1)$ by \cite[Theorem 1]{Ko}. 

Next, let the number $k$ of $\aut(\cX)$-short orbits equals $3$.  

First, assume $|N_{\aut(\cX)}(G)| = 2(g+1)$. Then either case (1) or (3) as in Proposition \ref{prop:involution} occurs. If case (1) occurs, then $\{P_0\}, \{P_\infty\}$ and $\{P_1,P_a\}$ are contained in distinct $\aut(\cX)$-short orbits. As the stabilizer of each of such points has either size $2(g+1)$ or $(g+1)$, a computation via the Riemann-Hurwitz formula shows that $\aut(\cX) = 2(g+1)$. If case (3) occurs, we may assume that $\{P_1,P_a\}$ and $\{P_0, P_\infty\}$ are contained in two distinct short orbits. Let $Q$ be a point in the third short orbit; then $e_Q = 2l$ for an integer $l \geq 1$. We may write the Riemann-Hurwitz formula applied to $\cX\rightarrow \cX/\aut(\cX)$ as 
$$
2g-2 \geq |\aut(\cX)| \Biggl(\frac{1}{2}- \frac{2}{g+1}\Biggr).
$$
A computation gives $|\aut(\cX)| \leq 12(g+1)$, whence $g \in \{4,6,10\}$. For $g= 10$, we must have $|\aut(\cX)| =  12(g+1)=132$, a contradiction since a group of order 11 should be normal in such a group (as $12$ is not a prime power). For $g =6$, let $n_7$ be the number of Sylow $7$-subgroups of $\aut(\cX)$. Then $n_7 \geq 8$, $n_7 \equiv 1 \mod 7$ and $n_7 = |\aut(\cX) : N_{\aut(\cX)}(G)|$ by the Sylow Theorem. However, this yields $|\aut(\cX)| \geq 112 > 84 =12(g+1)$, a contradiction.

If $g=4$, let $n_5$ be the number of Sylow $5$-subgroups of $\aut(\cX)$. Then $n_5 \geq 6$, $n_5 \equiv 1 \mod 4$ and $n_5= |\aut(\cX) : N_{\aut(\cX)}(G)|$ by the Sylow Theorem. This imply $\aut(\cX) = 60$, and then $\aut(\cX) \cong {\rm Alt}_5$, the alternating group on $5$ letters, contradicting Theorem \ref{thmgenuseven}. 

Finally, let $|N_{\aut(\cX)}(G)| = 4(g+1)$; we use the same notation as in Proposition \ref{4malizer}. Note that $N_{\aut(\cX)}(G)$ is a non-abelian group of order $4(g+1)$, $g \equiv 0 \mod 4$ and  $N_{\aut(\cX)}(G)$ is a Frobenius group. Since $p \neq 2$, we may apply Accola's result (see \cite[Lemma, page 477]{Ac}). Then $\delta$ fixes exactly two points outside $\Omega$, and the ramification index $e$ of a point in $\Omega$ in the cover $\cX\ \rightarrow \cX/\aut(\cX)$ equals $g+1$. By the proof of the Hurwitz bound, (see item (III) in the proof of \cite[Theorem 11.56]{hirschfeld-korchmaros-torres2008}), we have $|\aut(\cX)| \leq 40(g-1)$. Since by the Sylow Theorem  $|\aut(\cX)| \geq 4(g+1)(g+2)$, this gives a contradiction whenever $g > 4$.

Things being so, let $g = 4$; since $p \neq 2,3$, $\aut(\cX)$ is tame. Recall that a tame automorphism group  of a genus $4$-curve has size at most $120$. By Proposition \ref{4malizer},  $\cX$ is one of the following curves:
$$
\mathcal{Q}_1: \:\:Y^5 = X(X-1)^4(X-\ha)^2,
$$
$$
\mathcal{Q}_2: \:\:Y^5 = X(X-1)^4(X-\ha)^3,
$$
$$
\mathcal{Q}_3: \:\:Y^5 = X^2(X-1)^3(X-\ha), 
$$
$$
\: \:\mathcal{Q}_4: \:\:Y^5 = X^2(X-1)^3(X-\ha)^4. 
$$
It can be checked that all these curves are birationally equivalent among themselves. For instance, the map 
$$
\varphi (x,y)= \Bigl(-(x-1), \frac{y^3}{(x-1)^2(x-\ha)}\Bigr)
$$ 
is a birational map between $K(\mathcal{Q}_1)$ and $K(\mathcal{Q}_3)$; the other birational equivalences  are obtained in a similar fashion. Since $\aut(\cX)$ is tame and large, (that is, $\aut(\cX) > 4(g-1)$), we may look at the tables in \cite[Table 4]{maagard}, where large tame automorphism groups of curves of genus $4$ as well as their orbit-behavior are exhibited. We see that either $|\aut(\cX)| =20$ or $|\aut(\cX)| =120$, and if the latter holds, $\aut(\cX) \cong {\rm S}_5$. We can exclude the former possibility since in this case, $\aut(\cX)$ would be isomorphic to a dihedral group $D_{10}$ of order $20$, which does not contain any cyclic group of order 4. Hence, $\aut(\cX) \cong {\rm S}_5$, and our assertion follows.

\end{proof}

\begin{rem}\label{rem:finalmente}
The results in this subsection hold true whenever we replace $\aut(\cX)$ by an automorphism group $G'$ of $\cX$ containing $G$. In particular, by Corollary \ref{biggroup} and Theorem \ref{teoremchar0g+1}, if $G$ is not normal in $\aut(\cX)$, then $\aut(\cX)$ is non-tame and case (d) of Theorem \ref{th:largeaut} holds.
\end{rem}

\begin{lemma}\label{lematecnico}
For $p >3$, let $\cX$ be non-hyperelliptic $(g+1)$-curve of genus $g>2$ such that $G$ is not normal in $\aut(\cX)$. Then the following hold.

\begin{itemize}
\item[(i)] There is a cyclic automorphism $\tau \in \aut(\cX)$ of order $4$ normalizing $\alpha$ and acting transitively on $\Omega=\{P_0,P_1,P_a,P_\infty\}$.
\item[(ii)] $\aut(\cX)_{P_i}=G$ for all $i=0,1,a,\infty$.
\item[(iii)] There is no involution in $\aut(\cX)$ that commutes with $\alpha$. In particular, the group $\langle \tau,\alpha\rangle$ is neither isomorphic to $\mathbb{Z}_{4(g+1)}$ nor to $D_{2(g+1)}$.
\item[(iv)] $|\aut(\cX)|=4(g+1)(d(g+1)+1)$ for some positive integer $d$.
\end{itemize}
\end{lemma}
\begin{proof}
By Remark \ref{rem:finalmente}, $\aut(\cX)$ has two short orbits on $\cX$: one tame, say $\Lambda_1$, and one non-tame, say $\Lambda_2$. Moreover, Proposition \ref{tamenontame} implies $\Omega=\{P_0,P_1,P_a,P_\infty\} \subset \Lambda_1$. In particular, there must exist $\tau \in \aut(\cX)$ such that $\tau(P_0)=P_1$. Thus $\tau^{-1}\alpha \tau(P_0)=\tau^{-1}(P_1)=P_0$. Hence $\tau^{-1} \alpha \tau \in \aut(\cX)_{P_0}$, and then $\tau^{-1} \alpha \tau=\alpha$. Thus the order of $\tau$ is at most $4$. By Proposition \ref{prop:c3} and \cite[Theorem 1]{Ko} we have that $\tau$ has order $2$ or $4$. If $\tau$ is of order $2$, then there exists $\eta \in \aut(\cX)$ such that $\eta(P_0)=P_a$ and $\eta$ normalizes $\alpha$. Thus $\langle \tau,\eta \rangle$ is isomorphic to a subgroup of the symmetric group ${\rm S}_4$ without an element of order $3$, and then  $V_4 \ltimes G \leq \aut(\cX)$, contradicting Proposition \ref{prop:V4}. Hence $\tau$ has order $4$, and a similar argument shows that the action of $\tau$ on $\Omega$ must be transitive. This proves (i). 

To prove (ii), assume that $|\aut(\cX)_{P_i}| > g+1$. By \cite[Theorem 11.79]{hirschfeld-korchmaros-torres2008} we have $|\aut(\cX)_{P_i}|=k(g+1)$, with $k=2,3$ or $4$. But $\aut(\cX)_{P_i}$ is cyclic, and then $k \neq 3$, as shown in the previous paragraph. So let $\sigma \in \aut(\cX)_{P_i}$ of order $2(g+1)$, with $\sigma^2=\alpha$. Hence $\tau^2$ normalizes $\sigma$ and $\langle \tau^{2},\sigma\rangle/G$ is a non-cyclic subgroup of $N/G$ of order $4$, contradicting  Proposition \ref{prop:V4}.

Now suppose that $\mu$ is an involution commuting with $\alpha$. Set $\bar{\cX}=\cX/\langle \mu \rangle$ and $\bar{g}=g(\cX/\langle \mu \rangle)$. Via the Riemann-Hurwitz formula applied to the cover $\cX \rightarrow \bar{\cX}$, we conclude that $\mu$ has exactly two fixed points, both outside $\Omega$. But $\alpha \mu=\mu \alpha$, and therefore $\alpha$ must act on the set of fixed points of $\mu$. Hence $\alpha$ must fix such points, a contradiction since the set of fixed points of $\alpha$ is $\Omega$. This concludes the proof of (iii). 

Finally, since $\Omega \subset \Lambda_1$, then (i) and (ii) imply that $4$ divides $|\Lambda_1|$ and $|\Lambda_1|\equiv 4 \mod g+1$. Thus $|\Lambda_1|=4d(g+1)+4$ for some positive integer $d$. This combined with the Orbit-Stabilizer Theorem finishes the proof.
\end{proof}

\begin{lemma}\label{oddcorefree}
For $p >3$, let $\cX$ be non-hyperelliptic $(g+1)$-curve of genus $g>4$ such that $G$ is not normal in $\aut(\cX)$.  
 If $O(\aut(\cX))$ is non-trivial, then $G \leq O(\aut(\cX))$. 
\end{lemma}
\begin{proof}
Assume that there exists $H \triangleleft \aut(\cX)$ of odd order, with $\alpha \notin H$. In particular, $H \cap \langle \tau,\alpha \rangle=\{1\}$. Set $\tilde{\cX}=\cX/H$ and $\tilde{g}=g(\tilde{\cX})$. If $\tilde{g}=0$, then $\tilde{\cX}$ is a rational curve with an automorphism subgroup isomorphic to $\langle \tau,\alpha \rangle$. But this is impossible by Lemma \ref{lematecnico} (iii) and \cite[Theorem 1]{vmhaup}.

Suppose $\tilde{g}=1$. Then $G$ must be isomorphic to a subgroup of a non-trivial one-point stabilizer in the automorphism group of an elliptic curve, a contradiction to \cite[Theorem 11.94]{hirschfeld-korchmaros-torres2008}.

Therefore $\tilde{g} \geq 2$. In this case, we obtain that $\tilde{\cX}$ is a $(2\tilde{g}+1)$-curve, with $g+1=2\tilde{g}+1$. Denoting $d=|H|$, the Riemann-Hurwitz formula applied to the cover $\cX \rightarrow \tilde{\cX}$ provides that $d=2$, contradicting the fact that $H$ has odd order. 
\end{proof}

\begin{proposition}\label{oddcoretrivial}
For $p >3$, let $\cX$ be non-hyperelliptic $(g+1)$-curve of genus $g>4$ such $G$ is not normal in $\aut(\cX)$. Then $O(\aut(\cX))$ is trivial. 
\end{proposition}
\begin{proof} By contradiction, let  $O(\aut(\cX))$ be non-trivial. By Lemma \ref{oddcorefree}, $G \leq O(\aut(\cX))$. Then either $G$ is normal in $O(\aut(\cX))$ or not. If the former holds, then the only possibility is $|O(\aut(\cX))| = 3(g+1)$ and $\cX$ is isomorphic to the curve $\cF$ as in Proposition \ref{prop:c3}, a contradiction since by \cite[Theorem 1]{Ko}, $G$ is normal in $\aut(\cF)$. If the latter holds, then by Remark \ref{rem:finalmente} $O(\aut(\cX))$ is non-tame and we may apply Lemma \ref{lematecnico}(iii), a contradiction. 
\end{proof}

\begin{theorem}\label{hurwitzholds}
For $p >3$, let $\cX$ be non-hyperelliptic $(g+1)$-curve of genus $g>4$. Then $G$ is normal in  $\aut(\cX)$.
\end{theorem}
\begin{proof}
Since $g$ is even, the group structure of $\aut(\cX)$ is one of those given in Theorem \ref{thmgenuseven}. By contradiction, let $G$ be non-normal in  $\aut(\cX)$. Then \ref{lematecnico} applies. Also, by  Proposition \ref{oddcoretrivial}, $O(\aut(\cX))$  is trivial. We proceed by a case-by-case analysis based on Theorem \ref{thmgenuseven}. 
\begin{itemize}
\item By Lemma \ref{lematecnico} (i) $\aut(\cX)$ has even order, so (i) cannot occur. 

\item We clearly can rule out case (ii) as well, since $g+1$ divides $|\aut(\cX)|$. 

\item Assume that the commutator subgroup of $\aut(\cX)$ is isomorphic to $\SL(2,q)$ with $q \geq 5$. Suppose that $\alpha \notin \SL(2,q)$. An argument similar to the one used in the proof of Lemma \ref{oddcorefree} gives that $\cX/\SL(2,q)$ is rational. Let $\tau \in \aut(\cX)$ as in Lemma \ref{lematecnico}(i). We claim that $\tau \in \SL(2,q)$. Indeed, suppose that $\tau \not\in \SL(2,q)$. Then since $\aut(\cX)/\SL(2,q)$ is abelian, we have $\alpha^{-1}\tau^{-1}\alpha \tau \in \SL(2,q)$. But $\tau^{-1}\alpha \tau=\alpha^{j}$ for some $j$, which gives $\alpha^{j-1} \in \SL(2,q)$, thus $j=1$. In turn, this implies that $\tau$ commutes with $\alpha$, contradicting Lemma \ref{lematecnico}(iii). Hence $P_0,P_1,P_a$, and $P_\infty$ belong to the same orbit of $\SL(2,q)$, say $\Gamma_1$. Denote by $\tilde{\alpha} \in \aut(\cX/\SL(2,q))$ the automorphism induced by $\alpha$ and let $Q \in \cX/\SL(2,q)$ the point below $\Gamma_1$. Then $Q$ is fixed by $\tilde{\alpha}$. Since $\tilde{\alpha}$ fixes another point $R \in \cX/\SL(2,q)$, we have that the orbit $\Gamma_2$ lying over $R$ is preserved by $\alpha$. Thus $g+1$ divides $|\Gamma_2|$, as $\alpha$ fixes no point outside $\Omega$. On the other hand, it follows from Lemma \ref{lematecnico}(iv) that 
$$
|\SL(2,q)| \text{ divides } 4d(g+1)+4,
$$    
and then $r|\Gamma_2|=4d(g+1)+4$ for some $r$. Therefore $(g+1) | 4$, a contradiction.

Now assume that $\alpha \in \SL(2,q)$. If $q$ is odd, then $\SL(2,q)$ has a central involution, which is not possible in our case. Suppose then $q=2^w$ for some positive integer $w$, and so $\SL(2,q)=\PSL(2,q)$. Then $g+1$ divides $|\PSL(2,2^w)|=2^w(2^{2w}-1)$, which gives $|\PSL(2,2^w)|>84(g-1)$ since we are assuming $g>4$. Therefore, since all the results obtained in this section for $\aut(\cX)$ hold for subgroups of $\aut(\cX)$ containing $\alpha$, we conclude from Lemma \ref{lematecnico}(i) that there is a cyclic group of order $4$ in $\PSL(2,2^w)$. But this contradicts Dickson's classification of subgroups of $\PSL(2,q)$, see \cite[Theorem 4]{vmhaup}. Hence (iii) of Theorem \ref{thmgenuseven} cannot occur.

\item Suppose that $\PSL(2,q) \leq \aut(\cX) \leq {\rm P\Gamma L}(2,q)$, with $q \geq 3$. If $q=\ell^h$ for a prime $\ell$, recall that
$$
{\rm P\Gamma L}(2,q) \cong \PSL(2,q) \rtimes \gal(\mathbb{F}_q/\mathbb{F}_\ell),
$$
and then $|{\rm P\Gamma L}(2,q)|=\frac{hq(q^2-1)}{2}$. If $\alpha \notin \PSL(2,q)$, then $(g+1) |h$, which gives $q=\ell^{(g+1)c}$ for some positive integer $c$. In particular, $q>\ell(g+1) \geq 6(g+1)$, whence $|\aut(\cX)|>108g^4$, which is impossible by \cite[Hauptsatz]{St1973}.

Hence assume that $\alpha \in \PSL(2,q)$. Since $|\PSL(2,q)|=\frac{q(q^2-1)}{2}$, we have that $g+1$ divides either $q$, $q-1$ or $q+1$. If $(g+1) | q$, then $q=(g+1)^s$ for some positive integer $s$. Note that $s>1$ implies $|\aut(cX)| \geq g^6/2$, a contradiction by  \cite[Hauptsatz]{St1973} and \cite[Theorem 11.127]{hirschfeld-korchmaros-torres2008}. So $s=1$ and $g+1=q$. From the structure of $\PSL(2,q)$, there is an automorphism of order $g/2$ normalizing $\alpha$, and then $g \leq 8$. However, $g \neq 8$ since $g+1$ is prime and $g=6$ contradicts Proposition \ref{prop:c3}.

If $(g+1) | q-1$, we have $q-1=k(g+1)$ for some positive integer $k$. If $k$ is odd, then $q$ is even and so $\alpha \in \PSL(2,2^w)$ for some $w$. But we have already seen that such a case can be dismissed. For $k \geq 4$ we obtain $|\aut(\cX)|>16g^3>8g^3$, and \cite[Theorem 11.127]{hirschfeld-korchmaros-torres2008} implies that $\cX$ is isomorphic to the Hermitian curve, which is impossible as such a curve is not a $g+1$-curve. So we are left with the case $k=2$, that is, $2(g+1)=q-1$. If $\aut(\cX)>\PSL(2,q)$, then again $|\aut(\cX)|>8g^3$ gives a contradiction. Thus $\aut(\cX)=\PSL(2,q)$. From Lemma \ref{lematecnico}(i) there is $\tau \in \PSL(2,q)$ of order $4$ normalizing $\alpha$, and $\langle \tau, \alpha \rangle \leq \PSL(2,q)$ is a subgroup of order $2(q-1)$. However,  by \cite[Theorem 4]{vmhaup} there is no such subgroup in $\PSL(2,q)$.

The case $(g+1) | q+1$ can be excluded in the same way. Hence we can rule out (iv) of Theorem \ref{thmgenuseven}.

\item If $\PSL(3,q) \leq \aut(\cX) \leq {\rm P\Gamma L}(3,q)$ with $q \equiv 3 \mod 4$, using the same argument of the previous item, we would have $|\aut(\cX)|>g^8$ if $\alpha \notin \PSL(3,q)$, which is impossible, and $|\aut(\cX)|>8g^3$ if $\alpha \in \PSL(3,q)$ (which also leads us to a contradiction), except for $q=3$ and $g=12$. Since $|\PSL(3,q)|=\frac{q^3(q^3-1)(q^2-1)}{\gcd(3,q-1)}$, we obtain $|\PSL(3,3)|=2^4\cdot2^3\cdot13$, and then $p \leq 3$. This rules out (v) of Theorem \ref{thmgenuseven}. In the same way, with no exception, one may exclude case (vi), since  $|\PSU(3,q)|=\frac{q^3(q^3+1)(q^2-1)}{\gcd(3,q+1)}$.
\item Suppose $\aut(\cX)\cong \Alt_7$. As $|{\rm Alt}_7| = 2520 = 2^3\cdot3^2\cdot5\cdot7$, it follows that $ g+1 =7$, a contradiction as $2520 > 8\cdot 6^3=8g^3$. Note that this argument also excludes (x) of Theorem \ref{thmgenuseven}.
\item In case $\aut(\cX)\cong \rm{M}_{11}$, since $|\rm{M}_{11}| =  2^4\cdot 3^2 \cdot5\cdot 11 = 7920$, we have $ g+1 =11$. Such case cannot occur as the normalizer in $\rm{M}_{11}$ of a Sylow $11$-subgroup is  the semi-direct product $C_{11} \rtimes C_5$ where  $C_5$ is a cyclic $5$-group (see \cite{m11}), a contradiction to Remark \ref{remark:s4}. 
\item Both groups in cases (ix) and (xi) have order $48$, so these cases can be excluded immediately. 
\end{itemize}
\end{proof}

{\bf Proof of Theorem \ref{resumog+1}.}  An immediate consequence of  Theorems \ref{teoremchar0g+1} and \ref{hurwitzholds} and Corollary \ref{cor:4.14}. 
\qed

\subsection{Automorphism groups of wild $(g+1)$-curves}

In this subsection, we turn our attention to the automorphism group of wild non-hyperelliptic $(g+1)$-curves. By Theorem \ref{thm:p = g+1}, we know that a wild non-hyperelliptic $(g+1)$-curve admits a plane model 
$$
 \cZ_{d,e,\ell}: Y^{g+1}-Y = X^3+dX^2+eX+\ell,
$$
for $d,e,\ell \in K$. Since this model is given by separated polynomials, we may apply the results in \cite[Section 12.1]{hirschfeld-korchmaros-torres2008} and \cite{bonini}. Before doing so, it is useful to establish the different isomorphism classes of wild $(g+1)$-curves.

\begin{proposition}\label{isowild}
Let $B(X) = X^3+dX^2+eX+\ell$. Then the wild $(g+1)$-curve with affine equation $ Y^{g+1}-Y = B(X)$ is isomorphic to one of the following curves:
\begin{itemize}
\item $\cC :  Y^{g+1}-Y = X^3$;
\item $\cD: Y^{g+1}-Y = X(X-1)^2$;
\item $\cE_\lambda: Y^{g+1}-Y = X(X-1)(X-\lambda)$, for $\lambda \in K\setminus\{0,1\}$. 
\end{itemize}
\end{proposition}
\begin{proof} First, let $B(X)$ have only one root in $K$, that is, $B(X) = (X-a)^3$, for $a \in K$. A linear substitution in $X$ provides the required isomorphism between the curve with affine equation $Y^{g+1}+Y = B(X)$ and $\cC$. 

Next, let $B(X)$ have two distinct roots $x_0,x_1$ in $K$. Since $\PGL(2,K)$ is sharply $3$-transitive on $\mathbb{P}^1$, we may assume that $x_0 =0, x_1 =1$, whence the curve with affine equation $Y^{g+1}+Y = B(X)$ is isomorphic to the curve $\cD$. 

Finally, let $B(X)$ have three distinct roots $x_0,x_1,x_2$.  We distinguish two cases. First, let $B(X) = X^3-\ell$, for $\ell \in K$. Then the curve with affine equation $Y^{g+1}+Y = X^3-\ell$ is isomorphic to $\cC$ via the map $(x,y) \mapsto (x,y+m)$ where $m^p-m = \ell$.  If $B(X) \neq  X^3-\ell$, that is, there is no cyclic automorphism group of order three permuting $x_0,x_1,x_2$, we may assume that $x_0 = 0, x_1 =1$, again by the $3$-transitivity of $\PGL(2,K)$. Since $P_\infty$, (the common pole of $x,y$) has already been chosen, we have to let $x_2 =\lambda$ for $\lambda \in K\setminus\{0,1\}$, thus obtaining the equation of $\cE_\lambda$, and we are done.
\end{proof}

Thus, we may obtain the main result of this section.

\begin{theorem}
The automorphism group $\aut(\cZ_{d,e,\ell})$ of a wild $(g+1)$-curve fixes $P_\infty$, unless $g+1 \equiv -1 \mod 3$ and $\cZ_{d,e,\ell}$ is projectively equivalent to the curve 

$$
\cC : Y^{g+1}-Y = X^3, \: \mbox{with} \: |\aut(\cC_1)| = 3g(g+1)(g+2). 
$$

Further, if $g \equiv 0 \mod 3$, then $|\aut(\cC)| = 3g(g+1)$. In the remaining cases, $|\aut(\cZ_{d,e,\ell})| \in \{g+1,2(g+1),3(g+1)\}$. 

\end{theorem}
\begin{proof} 
Let $K(x,y) = K(\cZ_{d,e,\ell})$. Also, let $P_\infty$ be the common pole of $x$ and $y$ in $K(\cZ_{d,e,\ell})$. By \cite[Theorem 11.12]{hirschfeld-korchmaros-torres2008}, $\aut(\cZ_{d,e,\ell})$ fixes $P_\infty$ unless $\cZ_{d,e,\ell}$ is projectively equivalent either to the curve $\cC = \cZ_{0,0,0}$ with affine equation 

$$
\cC: Y^{g+1}-Y = X^3,
$$

with $g+1 \equiv -1 \mod 3$, or to an Hermitian curve. As the latter case is easily discarded, we may infer that $\aut(\cZ_{d,e,\ell}) = \aut(\cZ_{d,e,\ell})^{(0)}_{P_\infty}$ unless $g \equiv 1 \mod 3$, and $B(X)$ has exactly one root in $K$. If this is the case, always by \cite[Theorem 11.12]{hirschfeld-korchmaros-torres2008}, $\aut(\cZ_{d,e,\ell})$ contains a normal subgroup $C_3$ of order $3$ such that $\aut(\cZ_{d,e,\ell})/C_3 \cong \PGL(2,g+1)$. 
If $g  \equiv 0 \mod 3$, then by \cite[Theorem 3.2 (ii)]{bonini}, then $\aut(\cC) = G  \rtimes H$, where $H$ is a cyclic tame group with $|H| = 3g$. 

Next, assume that $\cZ_{d,e,\ell}$ is not isomorphic to $\cC$. By Proposition \ref{isowild}, then $\cZ_{d,e,\ell}$ is isomorphic either to $\cD$ or $\cE_\lambda$. Also, in both of these cases, $\aut(\cZ_{d,e,\ell}) =  G  \rtimes H$, where $H$ is a cyclic tame group whose size we need to determine. If the former holds, then by \cite[Remark 3.5]{bonini}, then the order of  $H$ is either $1$ or $2$. Let $H = \langle \gamma \rangle$;  by \cite[Proposition 3.4]{bonini}, $\gamma(x) = bx+c$, for $b, c \in K$, and $\gamma(B(x)) = aB(x)$, for $a \in \mathbb{F}_{(g+1)}^*$. A computation gives $b=1, c = 0$, whence $H$ is trivial.

If the latter holds, that is, $\cZ_{d,e,\ell}$ is isomorphic to $\cE_\lambda$, we first observe that $H$ must act on the set $\{0,1,\lambda\}$, and it is then isomorphic to a cyclic subgroup of $S_3$, whence $|H| \in \{1,2,3\}$. Again, for a generator $\gamma$ of $H$, we have $\gamma(x) = bx+c$, for $b, c \in K$, and $\gamma(B(x)) = aB(x)$, for $a \in \mathbb{F}_{(g+1)}^*$. A computation shows that $b =-1$ if $|H| =2$, or $b$ is a primitive $3$-rd root of the unity if $|H| = 3$. If the former holds, then $c = \lambda$ if $\lambda =2$, $c =0$ if $\lambda = -1$, $c = 1$ if $\lambda = 2^{-1}$. If the latter holds, then  $\lambda$ must be a primitive root of $-1$. 
\end{proof}

\section{Some remarks on $g$ and $(g-1)$-curves}\label{lower}
In this section, we give some partial results and some remarks regarding the classification and the  determination of the full automorphism group of $g$ and $(g-1)$-curves.

\subsection{On $g$-curves}
In this subsection, we focus on $g$-curves. The following result is obtained. 
\begin{proposition}
Let $\cX$ be a $g$-curve, and let  $\alpha \in \aut(\cX)$ of order $g$ with $G = \langle \alpha\rangle$.
Then one of the following holds. 
\begin{enumerate}
\item[{\rm (i)}] $g(\cX/G ) = 1$, with $\rho(\alpha) =2$ if $g \neq p$ and $\rho(\alpha) = 1$ if $g = p$;
\item[{\rm (ii)}] $g(\cX/G ) = 0$, with $g = 3$, $\rho(\alpha) =5$ if $g \neq p$ and $\rho(\alpha) = 2$ if $g = p$. 
\end{enumerate}
\end{proposition}
\begin{proof} Since $g$ is prime, regardless of which case holds ($ p = g$ or $p \neq g$), the Riemann-Hurwitz genus formula applied to the covering $\cX \rightarrow \cX/G$ reads 
\begin{equation}\label{eq:g=p}
2(g-1) = 2g(\bar{g}-1)+ s(g-1), 
\end{equation}
where $\bar{g } = g(\cX/G)$. From direct inspection, $\bar{g} \leq 1$. Assume that $\bar{g} =1$. Then Equation \eqref{eq:g=p} yields
$$
2(g-1) = + s(g-1).
$$
Hence, $s =2$. If $ g\neq p$, we have two distinct fixed points $P,Q$, whereas for $p =g$, we have a single point  fixed by $\alpha$ with $G_{P}^{(0)} = G_{P}^{(1)} = G$, while $G_{P}^{(2)}$ is trivial. 

Next, assume that $\bar{g} = 0$. In this case, Equation \eqref{eq:g=p} reads
$$
2(2g-1) = s(g-1). 
$$
The above equation admits integer solutions $s$ if and only if either $g =2$ or $g = 3$. As $g =2$ is excluded by our hypothesis, we are left with $q = g = 3$ and $s = 5$. If  $q \neq p$, we have then $5$ points fixed by $\alpha$, whereas if $q=p$ we have two points $P,Q$ fixed by $\alpha$, with $G_{P}^{(0)} = G_{P}^{(1)} = G$, while $G_{P}^{(2)}$ is trivial,  and $G_{Q}^{(0)} = G_{Q}^{(1)} = G_{Q}^{(2)} = G$, and $G_{Q}^{(3)}$ is trivial. 
\end{proof}

Thus, in the general case, to obtain explicit models for tame and wild $g$-curves, one should consider extensions of function fields of elliptic curves, with prescribed ramification. Assuming that a group of order $g$ in the automorphism group of a given $g$-curve is not self-normalizing, one may infer some more information on the underlying elliptic curve, and proceed to construct a model via Kummer or Artin-Schreier theory. We finally point out that without such a model, general results on the automorphism group of these curves (apart from something akin to the first results of Subsections 3.2 and 4.2) seem rather difficult to obtain, since these curves have odd genus. 

\subsection{On $(g-1)$-curves}

Finally, we turn our attention to $(g-1)$-curves. We prove the following. 

\begin{proposition}\label{first:g-1}

Let $\cX$ be a $(g-1)$-curve, and let  $\alpha \in \aut(\cX)$ of order $g-1$ with $G = \langle \alpha\rangle$. Then one of the following holds. 
\begin{enumerate}
\item[{\rm (i)}] $g(\cX/G ) = 2$, with $\rho(\alpha) =0$; 
\item[{\rm (ii)}] $g(\cX/G ) = 1$, with either $g = 3$ or $g = 4$;
\item[{\rm (iii)}] $g(\cX/G ) = 0$, with $g \in \{3,4,6\}$. 
\end{enumerate}
\end{proposition}
\begin{proof}
Let $G = \langle \alpha\rangle, \bar{g}= g(\cX/G)$.   As $g-1$ is prime, regardless of which case holds ($ p = g-1$ or $p \neq g-1$), the Riemann-Hurwitz genus formula applied to the covering $\cX \rightarrow \cX/G$ reads 
\begin{equation}\label{eq:g=p+1}
2(g-1) = 2(g-1)(\bar{g}-1)+ s(g-2). 
\end{equation}
From direct inspection, $\bar{g} \leq 2$. If $\bar{g} =2$, Equation \eqref{eq:g=p+1} reads 
$$
2(g-1) = 2(g-1) +s(g-2), 
$$
which is satisfied if, and only if, $s = 0$, and our first claim follows. 

If $\bar{g} =1$,  Equation \eqref{eq:g=p+1} reads 
$$
2(g-1) = s(g-2),
$$
whence either $g =3$ and $s =4$, or $g =4 $ and $s = 3$. 

If $\bar{g} =0$,  Equation \eqref{eq:g=p+1} reads 
$$
4(g-1) = s(g-2),
$$
whence either $g =3$ and $s =8$, or $g =4 $ and $s = 6$ or $g =6$ and $s =5$. 
\end{proof}

Let $\cX$ be $(g-1)$-curve such that $\cX/G$ is a curve of genus $2$. If ${\rm char}(K) =  g-1$, it is relatively \emph{easy} to provide explicit models of  $(g-1)$-curves via Artin-Schreier theory. More in detail, the techniques developed in \cite{storviana1989} to construct unramified Artin-Schreier extensions of function fields can be used once the quotient curve $\cX/G$ is given. 
However, it seems to us much more difficult to give examples over fields of characteristic different from $g-1$. 
The reason is that unramified Kummer extensions have not been so thoroughly studied in the literature. We provide the following example.

\begin{example}
Let $p\neq 5$. Then the plane curve with affine equation
$$
X^5Y^{10}=2Y^5+1
$$
is a 5-curve of genus $6$, whose function field was obtained through the following extension 
 $$
\begin{cases}
y^2 = x^5-1\\
z^5 = \frac{y+1}{x^5}\\
\end{cases}.
 $$
 of the function field of the genus two  curve with affine equation $Y^2 = X^5-1$. It can be checked through Magma that this curve has $150$ automorphisms, that is, it attains the maximum size of a tame automorphism group a genus $6$ curve. 
\end{example}

As $(g-1)$-curves have even genus, for $p \neq 2$ we may infer much information on their automorphism groups, again by using the results in \cite{giulietti-korchmaros-2017}. We point out that $(g-1)$-curves have recently been considered in the literature when $p=0$, see \cite{izquierdocarocca}. 

\begin{rem}
Let us recall that a Hurwitz curve is a curve attaining the Hurwitz upper bound $|G| = 84(g-1)$ for the size of a tame automorphism group $G$. Hurwitz curves do not exist for any genus; the smallest values are $g =3, 7,14,17$. For $g =3$, there exists only one such curves, namely the Klein quartic of affine equation $X^3Y +Y^3+X =0$. For  genus 7, again there is only one Hurwitz curve, the so-called Fricke-Machbeath curve, whose plane equation over $\mathbb{Q}$ is given  by $1+7XY+21X^2Y^2 +35X^3Y^3 +28X^4Y^4 +2X^7 +2Y^7 =0$. For higher genus, no explicit equation for a Hurwitz curve is known. 
For genus $14$, there are $3$ distinct Hurwitz curves, which happen to be $(g-1)$-curves. It seems possible, although non-trivial, to give explicit equations for such Hurwitz curves by building up on the results and observations made in this subsection. 
\end{rem}


\begin{thebibliography}{999}
\bibitem{Ac} R. D. M. Accola, Riemann Surfaces with automorphism groups admitting partitions, { \em Proceedings of the American Mathematical Society}, vol. 21, no. 2, 1969, pp. 477-482.

\bibitem{AS} N. Arakelian, P. Speziali, On generalizations of Fermat curves over finite fields and their automorphisms, {\em Comm. Algebra}, {\bf 45} (2017) no. 11, 4926-4938. 

\bibitem{bonini} M. Bonini, M. Montanucci, G. Zini, On plane curves given by separated polynomials and their automorphisms, \emph{Adv. Geom.}, {\bf 20} (2020), 61-70. 

\bibitem{MAGMA} W. Bosma, J. Cannon, and C. Playoust, The Magma algebra system. I. The user language, \emph{J. Symbolic Comput.}, {\bf 24} (1997), 235-265.

\bibitem{brandt} R. Brandt, \emph{ \"Uber die Automorphismengruppen von algebraischen Funktionenk\"orpern}, PhD thesis, Universit\"at-Gesamthochschule Essen, 1988. 

\bibitem{m11} T. Connor, D. Leemans, The subgroup lattice of $M_{11}$, available online at 
http://homepages.ulb.ac.be/~dleemans/atlaslat/m11.pdf 

\bibitem{giulietti-korchmaros-2017}
M. Giulietti, G. Korchm\'aros, Algebraic curves with many automorphisms, \emph{Adv. Math.}, {\bf 349} (2019) 162-211. 

\bibitem{hirschfeld-korchmaros-torres2008}
J.W.P. Hirschfeld, G. Korchm\'aros and F. Torres,\emph{  Algebraic Curves Over a Finite Field}, Princeton Univ. Press, Princeton, (2008), xiii + 720 pp.

\bibitem{homma1980}
M. Homma, Automorphisms of prime order of curves,  \emph{Manuscripta Math.} {\bf 33} (1980/81), no. 1, 99-109.

\bibitem{izquierdocarocca} M. Izquierdo, S. Reyes-Carocca, A note on large automorphism groups of compact Riemann surfaces, \emph{J. Algebra}, {\bf 547} (2020), 1-21. 

\bibitem{maagard} K. Magaard, T.Shaska, S. Shpectorov, and H. V\"olklein, The locus of curves with prescribed automorphism group, \emph{S$\bar{u}$rikaisekikenky$\bar{u}$sho K$\bar{o}$ky$\bar{u}$roku} {\bf 1267} (2002), 112-141. Communications in arithmetic fundamental groups (Kyoto, 1999/2001).

\bibitem{malmendier} A. Malmendier, T.Shaska, From hyperelliptic to superelliptic curves, \emph{Albanian J. Math.} {\bf 13} (2019), 107-200. 


\bibitem{Ko} A. Kontogeorgis, {\em The group of automorphisms of the function field of the curve $x^n+y^m+1=0$}, J. Number Theory {\bf 72} (1998), 110-136.

\bibitem{robinson} D. J.S. Robinson \emph{A Course in the Theory of Groups}, Springer-Verlag New York Berlin Heidelberg, (1996), xv+ 499 pp. 

\bibitem{Roq} P. Roquette   , Absch\"atzung der Automorphismen Anzahl von Funktionenk\"orpern bei Primzahl Charakteristik,  Math. Z. 117 (1970), 157-163.

\bibitem{seyama} A. Seyama, On the curves of genus $g$ with automorphisms of prime order $2g+1$, \emph{Tsukuba Journal of Mathematics}, {\bf 6} (1982), 67-77. 

\bibitem{shaska}
T.  Shaska, H. V\"olklein, Elliptic subfields and automorphisms of genus 2 function fields, \emph{Algebra, arithmetic and geometry with applications} (West Lafayette, IN, 2000), Springer, Berlin, (2004), 703-723.

\bibitem{schm} F. K. Schmidt, \emph{Zur arithmetischen Theorie der algebraischen Funktionen II. Allgemeine Theorie der Weierstra\ss punkte}, Math. Z. {\bf 45} (1939), 75-96.  

\bibitem{St1973} H.~Stichtenoth, \emph{\"Uber die Automorphismengruppe eines algebraischen Funktionenkörpers von Primzahlcharakteristik. I. Eine Abschätzung der Ordnung der Automorphismengruppe}, Arch. Math. 24 (1973), 527-544.

\bibitem{stbook}
H.~Stichtenoth,  \emph{Algebraic function fields and codes}, Springer-Verlag, Berlin and Heidelberg, (1993), vii+260 pp.

\bibitem{storviana1989}
K.O. St\"or, P. Viana, A Study of Hasse-Witt Matrices by local methods, \emph{Math. Z.} 200, 397-407 (1989). 


\bibitem{vmhaup}
R.C.Valentini, M. Madan, A Hauptsatz of L.E. Dickson and Artin-Schreier extensions, \emph{J. Reine Angew. Math.} {\bf 318} (1980), 156-177.
\bibitem{vdovin} E. P. Vdovin, Maximal orders of abelian subgroups in finite simple groups, \emph{Algebra and Logic}, {\bf 38} (n. 2) (1999), 67-83. 

\end{thebibliography}
\end{document}